\documentclass[aos,preprint]{imsart}

\RequirePackage[OT1]{fontenc}
\RequirePackage{amsthm,amsmath,graphicx}
\RequirePackage[sort,numbers]{natbib}
\RequirePackage[colorlinks,citecolor=red,urlcolor=blue]{hyperref}

\usepackage{amsfonts,amssymb}
\usepackage[ruled,vlined]{algorithm2e}
\usepackage[usenames,dvipsnames]{xcolor}

\arxiv{arXiv:0000.0000}

\startlocaldefs

\theoremstyle{plain}
\newtheorem{lemma}{Lemma}[section]
\newtheorem{proposition}{Proposition}[section]
\newtheorem{theorem}{Theorem}[section]

\newtheorem{remark}{Remark}[section]

\newcommand{\p} {\mathbb{P}}
\newcommand{\E} {\mathbb{E}}

\newcommand{\kl} {\text{KL}}
\newcommand{\mf} {\text{MF}}

\newcommand{\textmin} {\text{min}}
\newcommand{\textsum} {\text{sum}}

\newcommand{\norm}[1]{\left\|{#1} \right\|}

\newcommand{\opnorm}[1]{\|#1\|_{\rm op}}

\DeclareMathOperator*{\argmin}{arg\,min}
\DeclareMathOperator*{\argmax}{arg\,max}

\let\oldenumerate\enumerate
\renewcommand{\enumerate}{
  \oldenumerate
  \setlength{\itemsep}{1pt}
  \setlength{\parskip}{0pt}
  \setlength{\parsep}{0pt}
}

\newcommand{\Ber} {\text{Ber}} 

\newcommand{\tbeta}{\text{Beta}}
\newcommand{\tcate}{\text{Categorical}}

\newcommand{\zi}{z_i}
\newcommand{\zj}{z_j}

\newcommand{\Ztruth}{Z^*}
\newcommand{\Ztrutht}{Z^{*T}}

\newcommand{\tpr}{\text{pri}}
\newcommand{\pipr}{\pi^\tpr}
\newcommand{\alphapr}{\alpha^\tpr}
\newcommand{\betapr}{\beta^\tpr}

\newcommand{\tinit}{\text{init}}

\newcommand*\diff{\mathop{}\!\mathrm{d}}

\endlocaldefs

\begin{document}
\begin{frontmatter}
\title{Theoretical and Computational Guarantees \\ of Mean Field Variational Inference \\ for Community Detection}
\runtitle{Mean Field for Community Detection}

\begin{aug}
\author{\fnms{Anderson Y.} \snm{Zhang}\ead[label=e1]{ye.zhang@yale.edu}},
\and
\author{\fnms{Harrison H.} \snm{Zhou}\ead[label=e2]{huibin.zhou@yale.edu}\ead[label=u1,url]{http://www.stat.yale.edu/\textasciitilde hz68/}}
\runauthor{Zhang, Zhou}
\affiliation{Yale University}
\address{Department of Statistics\\
Yale University\\
New Haven, CT 06511 \\
\printead{e1}\\
\printead{e2}\\
\printead{u1}}
\end{aug}

\begin{abstract}
The mean field variational Bayes method is becoming increasingly popular in statistics and machine learning. Its iterative Coordinate Ascent Variational Inference algorithm has been widely applied to large scale Bayesian inference.  See Blei et al. (2017) for a recent comprehensive review. Despite the popularity of the mean field method there exist remarkably little fundamental theoretical justifications. To the best of our knowledge, the iterative algorithm has never been investigated for any high dimensional and complex model.  In this paper, we study the mean field method for community detection under the Stochastic Block Model. For an iterative Batch Coordinate Ascent Variational Inference algorithm, we show that it has a linear convergence rate and converges to the minimax rate within $\log n$ iterations. This complements the results of Bickel et al. (2013) which studied the global minimum of the mean field variational Bayes and obtained asymptotic normal estimation of global model parameters. In addition, we obtain similar optimality results for Gibbs sampling and an iterative procedure to calculate maximum likelihood estimation, which can be of independent interest.   

\end{abstract}



\end{frontmatter}


\section{Introduction}

A major challenge of large scale Bayesian inference is the calculation of posterior distribution.  For high dimensional and complex models, the exact calculation of posterior distribution is often computationally intractable. To address this challenge, the mean field variational method \cite{beal2003variational, jordan1999introduction, wainwright2008graphical} is used to approximate posterior distributions in a wide range of applications in many fields including natural language processing \cite{blei2003latent, liang2007infinite}, computational neuroscience \cite{grabska2017probabilistic, penny2005bayesian}, and network science \cite{airoldi2008mixed, celisse2012consistency, hofman2008bayesian}. This method is different from Markov chain Monte Carlo (MCMC) \cite{gelfand1990sampling, robert2004monte}, another popular approximation algorithm. The variational inference approximation is deterministic for each iterative update, while MCMC is a randomized sampling algorithm, so that for large-scale data analysis, the mean field variational Bayes usually converges faster than MCMC \cite{blei2017variational}, which is particularly attractive in the big data era.

In spite of a wide range of successful applications of the mean field variational Bayes, its fundamental theoretical properties are rarely investigated. The existing literature \cite{bickel2013asymptotic, celisse2012consistency, wang2006convergence, westling2015establishing, you2014variational} is mostly on low dimensional parameter estimation and on the global minimum of the variational Bayes method. For example, in a recent inspiring paper, Wang and Blei \cite{wang2017frequentist} studied the frequentist consistency of the variational method for a general class of latent variable models. They obtained consistency for low dimensional global parameters and further showed asymptotic normality, assuming the global minimum of the variational Bayes method can be achieved. However, it is often computationally infeasible to attain the global minimum when the model is high-dimensional or complex. This motivates us to investigate the statistical properties of the mean field in high dimensional settings, and more importantly, to understand the statistical and computational guarantees of the iterative variational inference algorithms.

The success and the 	popularity of the mean field method in Bayesian inference mainly lies in the success of its iterative algorithm: Coordinate Ascent Variational Inference (CAVI) \cite{blei2017variational}, which provides a computationally efficient way to approximate the posterior distribution. It is important to understand what statistical properties CAVI has and how do they compare to the optimal statistical accuracy. In addition, we want to investigate how fast CAVI converges for the purpose of implementation. With the ambition of establishing a universal theory of the mean field iterative algorithm for general models in mind, in this paper, we consider the community detection problem \cite{bickel2009nonparametric, mossel2012stochastic, newman2006modularity, airoldi2008mixed, gao2015achieving, zhang2016minimax} under the Stochastic Block Model (SBM) \cite{holland1983stochastic, bickel2009nonparametric, rohe2011spectral, lei2015consistency} as our first step.

Community detection has been an active research area in recent years, with the SBM as a popular choice of model. The Bayesian framework and the variational inference for community detection are considered in \cite{bickel2013asymptotic, gao2015general, airoldi2008mixed, celisse2012consistency, hofman2008bayesian, razaee2017matched}. For high dimensional settings, Celisse et al. \cite{celisse2012consistency} and Bickel et al. \cite{bickel2013asymptotic} are arguably the first to study the statistical properties of the mean field for SBMs. The authors built an interesting connection between full likelihood and variational likelihood, and then studied the closeness of maximum likelihood and maximum variational likelihood, from which they obtained consistency and asymptotic normality for global parameter estimation. From a personal communication with the authors of Bickel et al. \cite{bickel2013asymptotic}, an implication of their results is that the variational method achieves exact community recovery under a strong signal-to-noise (SNR) ratio. Their analysis idea is fascinating, but it is not clear whether it is possible to extend the analysis to other SNR conditions under which exact recovery may never be possible. More importantly, it may not be computationally feasible to maximize the variational likelihood for the SBM, as seen from Theorem \ref{thm:MF_simplifed}.

In this paper, we consider the statistical and computational guarantees of the iterative variational inference algorithm for community detection. The primary goal of community detection problem is to recover the community membership in a network. We measure the performance of the iterative variational inference algorithm by comparing its output with the ground truth. Denote the underlying ground truth by $\Ztruth$. For a network of $n$ nodes and $k$ communities, $\Ztruth$ is an $n\times k$ matrix with each row a standard Euclidean basis in $\mathbb{R}^k$. The index of non-zero coordinate of each row $\{\Ztruth_{i,\cdot}\}_{i=1}^n$ gives the community assignment information for the corresponding node.  We propose an iterative algorithm called Batch Coordinate Ascent Variational Inference (BCAVI), a slight modification of CAVI with batch updates, to make parallel and distributed computing possible.  Let $\pi^{(s)}$ denote the output of the $s$-th iteration, an $n\times k$ matrix with nonnegative entries. The summation of each row $\{\pi^{(s)}_{i,\cdot}\}_{i=1}^n$ is equal to 1, which is interpreted as an approximate posterior probability of assigning the corresponding node of each row into $k$ communities. The performance of $\pi^{(s)}$ is measured by an $\ell_1$ loss $\ell(\cdot,\cdot)$ compared with $\Ztruth$.
~\\
\newline
\textit{An Informal Statement of Main Result:} Let $\pi^{(s)}$ be the estimation of community membership from the iterative algorithm BCAVI after $s$ iterations. Under weak regularity condition, for some $c_n=o_n(1)$, with high probability, we have for all $s\geq 0$,
\begin{align}\label{eqn:informal}
\ell(\pi^{(s+1)},\Ztruth) \leq \text{minimax rate} + c_n\ell( \pi^{(s)},\Ztruth).
\end{align}
~\\
\indent 
The main contribution of this paper is Equation (\ref{eqn:informal}). The coefficient $c_n$ is $o_n(1)$ and is independent of $s$, which implies $\ell(\pi^{(s)},\Ztruth) $ decreases at a fast linear rate. In addition, we show that BCAVI converges to the statistical optimality \cite{zhang2016minimax}. It is worth mentioning that after $\log n$ iterations BCAVI attains the minimax rate, up to an error $o_n(n^{-a})$ for any constant $a>0$. The conditions required for the analysis of BCAVI are relatively mild. We allow the number of communities to grow. The sizes of the communities are not assumed to be of the same order. The separation condition on global parameters covers a wide range of settings from consistent community detection to exact recovery.  

To the best of our knowledge this provides arguably the first theoretical justification for the iterative algorithm of the mean field variational method in a high-dimensional and complex setting. Though we focus on the problem of community detection in this paper, we hope the analysis would shed some light on analyzing other models, which may eventually lead to a general framework of understanding the mean field theory. 

The techniques of analyzing the mean field can be extended to providing theoretical guarantees for other iterative algorithms, including Gibbs sampling and an iterative procedure for maximum likelihood estimation, which can be of independent interest.  Results similar to Equation (\ref{eqn:informal}) are obtained for both methods under the SBM.

\paragraph{Organization} The paper is organized as follows. In Section \ref{sec:mf_all} we introduce the mean field theory and the implementation of BCAVI algorithm for community detection.  All the theoretical justifications for the mean field method are in Section \ref{sec:main_theory}. Discussions on the convergence of the global minimizer and other iterative algorithms are presented in Section \ref{sec:diss}. The proofs of theorems are in Section \ref{sec:main_proof}. We include all the auxiliary lemmas and propositions and their corresponding proofs in the supplemental material.

\paragraph{Notation} Throughout this paper, for any matrix $X\in\mathbb{R}^{n\times m}$, its $\ell_1$ norm is defined in analogous to that of a vector. That is, $\norm{X}_1 =\sum_{i,j}|X_{i,j}|$. We use the notation $X_{i,\cdot}$ and $X_{\cdot,i}$ to indicate its $i$-th row and column respectively. For matrices $X,Y$ of the same dimension, their inner product is defined as $\langle X,Y \rangle=\sum_{i,j}X_{i,j}Y_{i,j}$. For any set $D$, we use $|D|$ for its cardinality. We denote $\Ber(p)$ for a Bernoulli random variable with success probability $p$. For two positive sequences $x_n$ and $y_n$, $x_n \lesssim y_n$ means $x_n \leq cy_n$ for some constant $c$ not depending on $n$. We adopt the notation $x_n\asymp y_n$ if $x_n \lesssim y_n$ and $y_n \lesssim x_n$. To distinguish from the probabilities $p,q$, we use bold $\mathbf{p}$ and $\mathbf{q}$ to indicate distributions. The Kullback-Leibler divergence between two distributions is defined as $\kl(\mathbf{p}\|\mathbf{q})=\E_\mathbf{q}\log (\mathbf{p}(x)/\mathbf{q}(x))$. We use $\psi(\cdot)$ for the digamma function, which is defined as the logarithmic derivative of Gamma function, i.e., $\psi(x)=   \frac{d}{dx}\left[\log \Gamma(x)\right]$. In any $\mathbb{R}^d$, we denote $\{e_a\}_{a=1}^d$ to be the standard Euclidean basis with $e_1=(1,0,0,\ldots), e_2=(0,1,0,\ldots,0),\ldots, e_d=(0,0,0,\ldots, 1)$. We let $1_d$ be a vector of length $d$ whose entries are all $1$. We use $[d]$ to indicate the set $\{1,2,\ldots,d\}$. Throughout this paper, the superscript ``pri'' (e.g., $\pipr$)  indicates that this is a hyperparameters of priors.

\section{Mean Field Method for Community Detection}\label{sec:mf_all}

In this section, we first give a brief introduction to the variational inference method in Section \ref{sec:general_mf}. Then we introduce the community detection problem and the Stochastic Block Model in Section \ref{sec:SBM}. The Bayesian framework is presented in Section \ref{sec:bayesian}. Its mean field approximation and CAVI updates are given in Section \ref{sec:mf} and Section \ref{sec:CAVI} respectively. The BCAVI algorithm is introduced in Section \ref{sec:BCAVI}.

\subsection{Mean Field Variational Inference}\label{sec:general_mf}
We first present the mean field method in a general setting and then consider its application to the community detection problem. Let $\mathbf{p}(x|y)$ be an arbitrary posterior distribution for $x$, given observation $y$. Here $x$ can be a vector of latent variables, with coordinates $\{x_i\}$. It may be difficult to compute the posterior $\mathbf{p}(x|y)$ exactly. The variational Bayes ignores the dependence among $\{x_i\}$, by simply taking a product measure $\mathbf{q}(x)=\prod_i \mathbf{q}_i(x_i)$ to approximate it. Usually each $\mathbf{q}_i(x_i)$ is simple and easy to compute. The best approximation is obtained by minimizing the Kullback-–Leibler divergence between $\mathbf{q}(x)$ and $\mathbf{p}(x|y)$:
\begin{align}\label{eqn:general_global_min}
\mathbf{\hat q}^\mf = \argmin_{\mathbf{q}\in\mathbf{Q}}\kl(\mathbf{q}\|\mathbf{p}).
\end{align}
Despite the fact that every measure $\mathbf{q}$ has a simple product structure, the global minimizer $\mathbf{\hat q}^\mf$ remains computationally intractable. 

To address this issue, an iterative Coordinate Ascent Variational Inference (CAVI) is widely used to approximate the global minimum. It is a greedy algorithm. The value of $\kl(\mathbf{q}\|\mathbf{p})$ decreases in each coordinate update:
\begin{align}\label{eqn:CAVI_update_general}
\mathbf{\hat q}_i = \min_{\mathbf{q}_i\in \mathbf{Q}_i}\kl\left[\mathbf{q}_i\prod_{j\neq i}\mathbf{q}_j\Bigg\|\mathbf{p}\right],\forall i.
\end{align}
The coordinate update has an explicit formula
\begin{align}\label{eqn:CAVI_update_general_explicit}
\mathbf{\hat q}_i(x_i) \propto\exp\left[\E_{\mathbf{q}_{-i}}\left[\log \mathbf{p}(x_i|x_{-i},y)\right]\right],
\end{align}
where $x_{-i}$ indicates all the coordinates in $x$ except $x_i$, and the expectation is over $\mathbf{q}_{-i}=\prod_{j\neq i}\mathbf{q}_j(x_j)$.  Equation (\ref{eqn:CAVI_update_general_explicit}) is usually easy to compute, which makes CAVI computationally attractive, although CAVI only guarantees to achieve a local minimum.  

In summary, the mean field variational inference via CAVI can be represented in the following diagram:
\begin{align*}
\mathbf{p}(x|y)\stackrel{\text{approx.}}{\Longleftarrow}\mathbf{\hat q}^\mf(x) \stackrel{\text{approx.}}{\Longleftarrow}\mathbf{\hat q}^\text{CAVI}(x),
\end{align*}
where $\mathbf{\hat q}^\mf(x)$, the global minimum, serves mainly as an intermediate step in the mean field methodology. What is implemented in practice to approximate global minimum is an iterative algorithm like CAVI. This motivates us to consider directly the theoretical guarantees of the iterative algorithm in this paper.

We refer the readers to a nice review and tutorial by Blei et al. \cite{blei2017variational} for more detail on the variational inference and CAVI. The derivation from Equation (\ref{eqn:CAVI_update_general}) to Equation (\ref{eqn:CAVI_update_general_explicit}) can be found in many variational inference literatures \cite{blei2017variational, bishop2006pattern}. We include it in Appendix \ref{sec:appendix_general} in the supplemental material for completeness.

\subsection{Community Detection and Stochastic Block Model}\label{sec:SBM}
The Stochastic Block Model (SBM) has been a popular model for community detection. 

Consider an $n$-node network with its adjacency matrix denoted by $A$. It is an unweighted and undirected network without self-loops, with $A\in\{0,1\}^{n\times n}$, $A=A^T$ and $A_{i,i}=0,\forall i\in[n]$. Each edge is an independent Bernoulli random variable with $\E A_{i,j}=P_{i,j},\forall i<j.$  In the SBM, the value of connectivity probability $P_{i,j}$ depends on the communities the two endpoints $i$ and $j$ belong to. We assume $P_{i,j}=p$ if both nodes come from the same community and $P_{i,j}=q$ otherwise. There are $k$ communities in the network. We denote $z\in[k]^n$, as the assignment vector, with $\zi$ indicating the index of community the $i$-th node belongs to. Thus, the connectivity probability matrix $P$ can be written as
\begin{align*}
P_{i,j}=B_{\zi,\zj},
\end{align*}
where $B\in [0,1]^{k\times k}$ with diagonal entries as $p$ and off-diagonal entries as $q$. That is, $B=q 1_k1_k^T+(p-q)I_k$. Let $Z\in\Pi_0$ be the assignment matrix where
\begin{align*}
\Pi_0= \{\pi\in\{0,1\}^{n\times k}:\norm{\pi_{i,\cdot}}_0=1,\forall i\in[n]\}.
\end{align*}
In each row $\{Z_{i,\cdot}\}_{i=1}^n$ there is only one 1 with all the other coordinates as 0, indicating the assignment of community for the corresponding node. Then $P$ can be equivalently written as $P_{i,j}=Z_{i,\cdot} B Z_{j,\cdot}^T,\forall i<j$, or in a matrix form
\begin{align*}
P_{i,j}=(ZBZ^T)_{i,j},\forall i<j.
\end{align*}

The goal of community detection is to recover the assignment vector $z$, or equivalently, the assignment matrix $Z$. The equivalence can be seen by observing that there is a bijection $r$ between $z\in[k]^n$ and $Z\in\Pi_0$ which is defined as follows,
\begin{align}\label{eqn:bijection_r}
 r(z)=Z\text{, where } Z_{i,a}=\mathbb{I}\{a=z_i\},\forall i\in[n],a\in[k].
\end{align}
Since they are uniquely determined by each other, in our paper we may use $z$ directly without explicitly defining $z=r^{-1}(Z)$ (or vice versa) when there is no ambiguity.

\subsection{A Bayesian Framework}\label{sec:bayesian}
Throughout the whole paper, we assume $k$, the number of communities, is known. We observe the adjacency matrix $A$. The global parameters $p$ and $q$ and the community assignment $Z$ are unknown. From the description of the model in Section \ref{sec:SBM}, we can write down the distribution of $A$ as follows:
\begin{align}\label{eqn:likelihood}
\mathbf{p}(A|Z,p,q) = \prod_{i<j}B_{z_i,z_j}^{A_{i,j}}(1-B_{z_i,z_j})^{1-A_{i,j}},
\end{align}
with $B=q1_k1_k^T+(p-q) I_k$ and $z=r^{-1}(Z)$. We are interested in Bayesian inference for estimating $Z$, with prior to be given on both $p,q$ and $Z$.

We assume that $\{z_{i}\}_{i=1}^n$ have independent categorical (a.k.a. multinomial with size one) priors with hyperparameters $\{\pipr_{i,\cdot}\}_{i=1}^n$, where $\sum_{a=1}^k \pipr_{i,a}=1,\forall i\in[n]$. In other words, $\{Z_{i,\cdot}\}_{i=1}^n$ are independently distributed by
\begin{align*}
\p(Z_{i,\cdot}= e_a^T) = \pipr_{i,a},\forall a=1,2,\ldots,k,
\end{align*}
where $\{e_a\}_{a=1}^k$ are the coordinate vectors. Here we allow the priors for $Z_{i,\cdot}$ to be different for different $i$. If additionally $\pi_{i,\cdot}=\pi_{j,\cdot}$ for all $i\neq j$ is assumed, and then this is reduced to the usual case of i.i.d. priors.

Since $\{A_{i,j}\}_{i<j}$ are Bernoulli, it is natural to consider a conjugate Beta prior for $p$ and $q$. Let $p\sim\tbeta(\alphapr_p,\betapr_p)$ and $q\sim\tbeta(\alphapr_q,\betapr_q)$. Then the joint distribution is
\begin{align}\label{eqn:joint_likelihood}
\mathbf{p}(A,Z,p,q)&=\left[\prod_i \pipr_{i,z_i}\right]\left[\prod_{i<j}B_{z_i,z_j}^{A_{i,j}}(1-B_{z_i,z_j})^{1-A_{i,j}}\right]\\
&\quad\times\left[\frac{\Gamma(\alphapr_p+\betapr_p)}{\Gamma(\alphapr_p)\Gamma(\betapr_p)}p^{\alphapr_p-1}(1-p)^{\betapr_p-1}\right] \left[\frac{\Gamma(\alphapr_q+\betapr_q)}{\Gamma(\alphapr_q)\Gamma(\betapr_q)}q^{\alphapr_q-1}(1-q)^{\betapr_q-1}\right].\nonumber
\end{align}
Our main interest is to infer $Z$, from the posterior distribution $\mathbf{p}(Z,p,q|A)$. However, the exact calculation of $\mathbf{p}(Z,p,q|A)$ is computationally intractable.

\subsection{Mean Field Approximation}\label{sec:mf}
Since the posterior distribution $\mathbf{p}(Z,p,q|A)$ is computationally intractable, we apply the mean field approximation to approximate it by a product measure,
\begin{align*}
\mathbf{q}_{\pi,\alpha_p,\beta_p,\alpha_q,\beta_q}(Z,p,q)=\mathbf{q}_{\pi}(Z)\mathbf{q}_{\alpha_p,\beta_p}(p)\mathbf{q}_{\alpha_q,\beta_q}(q)
\end{align*}
where $\{r^{-1}(Z_{i,\cdot})\}_{i=1}^n$ are independent categorical variables with parameters $\{\pi_{i,\cdot}\}_{i=1}^n$, i.e., $\mathbf{q}_{\pi}(Z)=\prod_{i=1}^n \mathbf{q}_{\pi_{i,\cdot}}(Z_{i,\cdot})$ with 
 \begin{align*}
 \mathbf{q}_{\pi_{i,\cdot}}(Z_{i,\cdot}=e_a) =\pi_{i,a},\forall i\in[n],a\in[k],
 \end{align*}
and $\mathbf{q}_{\alpha_p,\beta_p}(p)$ and $\mathbf{q}_{\alpha_q,\beta_q}(q)$ are Beta with parameters $\alpha_p,\beta_p,\alpha_q,\beta_q$ due to conjugacy. See Figure \ref{fig:plate} for the graphical presentation of $\mathbf{q}_{\pi,\alpha_p,\beta_p,\alpha_q,\beta_q}(Z,p,q)$. 

\begin{figure}[ht]
\includegraphics[width=\textwidth]{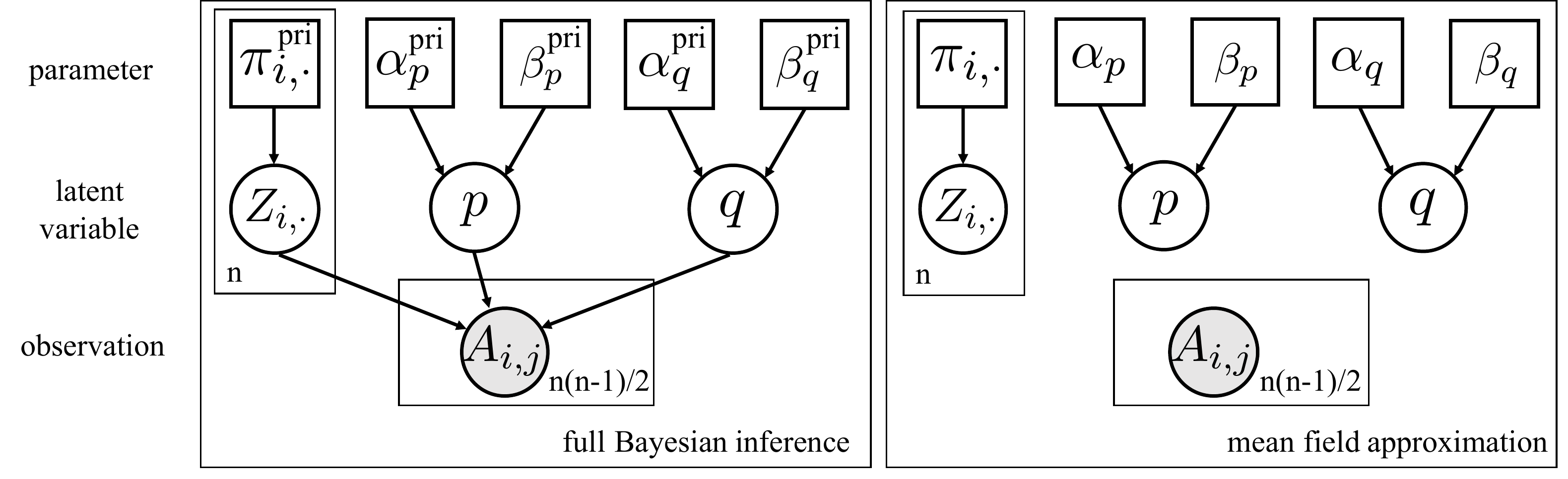}
\caption{Graphical model presentations of full Bayesian inference (left panel) and the mean field approximation (right panel) for community detection. The edges show the dependence among variables. \label{fig:plate}}
\end{figure}

Note that the distribution class of $\mathbf{q}$ is fully captured by the parameters $(\pi,\alpha_p,\beta_p,\alpha_q,\beta_q)$, and then the optimization in Equation (\ref{eqn:general_global_min}) is equivalent to minimize over the parameters as
\begin{align}
(\hat\pi^\mf,\hat\alpha_p^\mf,\hat\beta_p^\mf,\hat\alpha_q^\mf,\hat\beta_q^\mf) &= \argmin_{\substack{\pi\in\Pi_1 \\ \alpha_p,\beta_p,\alpha_q,\beta_q>0}}\kl\left[\mathbf{q}_{\pi,\alpha_p,\beta_p,\alpha_q,\beta_q}(Z,p,q)\Big\|\mathbf{p}(Z,p,q|A)\right],\label{eqn:MF}\\
\text{where }\Pi_1&=\{\pi\in[0,1]^{n\times k},\norm{\pi_{i,\cdot}}_1=1\}.\nonumber
\end{align}
 
Here $\Pi_1$ can be viewed as a relaxation of $\Pi_0$: it uses an $\ell_1$ constraint on each row instead of the $\ell_0$ constraint used in $\Pi_0$. The global minimizer $\mathbf{q}_{\hat\pi^\mf}(Z)$ gives approximate probabilities to classify every node to each community. The optimization in Equation (\ref{eqn:MF}) can be shown to be equivalent to a more explicit optimization as follows. Recall $\psi(\cdot)$ is the digamma function with $\psi(x)=\frac{\diff}{\diff x}[\log\Gamma(x)]$.

\begin{theorem}\label{thm:MF_simplifed}
The mean field estimator $(\hat\pi^\mf,\hat\alpha_p^\mf,\hat\beta_p^\mf,\hat\alpha_q^\mf,\hat\beta_q^\mf)$ defined in Equation (\ref{eqn:MF}) is equivalent to
\begin{align*}
(\hat\pi^\mf,\hat\alpha_p^\mf,\hat\beta_p^\mf,\hat\alpha_q^\mf,\hat\beta_q^\mf)= \argmin_{\substack{\pi\in\Pi_1 \\ \alpha_p,\beta_p,\alpha_q,\beta_q>0}} f(\pi,\alpha_p,\beta_p,\alpha_q,\beta_q; A),
\end{align*}
where
\begin{align*}
f(\pi,\alpha_p,\beta_p,\alpha_q,\beta_q; A) &= t\langle A-\lambda 1_n1_n^T+\lambda I_n,\pi\pi^{T}\rangle + \frac{1}{2}\left[\psi(\alpha_q)-\psi(\beta_q)\right]\norm{A}_1\\
&+ \frac{n}{2}\left[\psi(\beta_q)-\psi(\alpha_q+\beta_q)\right] -\sum_{i=1}^n\kl\left[\tcate(\pi_{i,\cdot})\|\tcate(\pipr_{i,\cdot})\right]\\
&-\kl\left[\tbeta(\alpha_p,\beta_p)\|\tbeta(\alphapr_p,\betapr_p)\right]- \kl\left[\tbeta(\alpha_q,\beta_q)\|\tbeta(\alphapr_q,\betapr_q)\right],
\end{align*}
and
\begin{align}
t&=\left[[\psi(\alpha_p)-\psi(\beta_p)] - [\psi(\alpha_q)-\psi(\beta_q)]\right]/2\label{eqn:t_formula}\\
\lambda &= \left[[\psi(\beta_q)-\psi(\alpha_q+\beta_q)]- [\psi(\beta_p)-\psi(\alpha_p+\beta_p)]\right]/(2t).\label{eqn:lambda_formula}
\end{align}
\end{theorem}

The explicit formulation in Theorem \ref{thm:MF_simplifed} is helpful to understand the global minimizer of the mean field method. However, the global minimizer $\hat\pi^\mf$ remains computationally infeasible as the objective function is not convex. Fortunately, there is a practically useful algorithm to approximate it.

\subsection{Coordinate Ascent Variational Inference}\label{sec:CAVI}

CAVI is possibly the most popular algorithm to approximate  the global minimum of the mean field variational Bayes.  It is an iterative algorithm. In Equation (\ref{eqn:MF}), there are latent variables $\{Z_{i,\cdot}\}_{i=1}^n,p,q$.  CAVI updates them one by one. Since the distribution class of $\mathbf{q}$ is uniquely determined by the parameters $\{\pi_{i,\cdot}\}_{i=1}^n,\alpha_p,\beta_p,\alpha_q,\beta_q$, equivalently we are updating those parameters iteratively. Theorem \ref{thm:coor_update} gives explicit formulas for the coordinate updates.

\begin{theorem}\label{thm:coor_update}
Starts with some $\pi,\alpha_p,\beta_p,\alpha_q,\beta_q$, the CAVI update for each coordinate (i.e., Equation (\ref{eqn:CAVI_update_general}) and Equation (\ref{eqn:CAVI_update_general_explicit})) has an explicit expression as follows:
\begin{itemize}
\item Update on $p$:
\begin{align*}
\alpha'_p = \alphapr_p + \sum_{i<j}\sum_{a=1}^k \pi_{i,a}\pi_{j,a}A_{i,j},\text{ and }\beta_p'=\betapr_p+\sum_{i<j}\sum_{a=1}^k \pi_{i,a}\pi_{j,a}(1-A_{i,j}).
\end{align*}
\item Update on $q$:
\begin{align*}
\alpha'_q = \alphapr_q + \sum_{i<j}\sum_{a\neq b}\pi_{i,a}\pi_{j,b}A_{i,j},\text{ and }\beta_q'=\betapr_q+\sum_{i<j}\sum_{a\neq b}\pi_{i,a}\pi_{j,b}(1-A_{i,j}).
\end{align*}
\item Update on $Z_{i,\cdot},\forall i=1,2,\ldots,n$:
\begin{align*}
\pi'_{i,a}\propto \pipr_{i,a}\exp\left[2t\sum_{j\neq i}\pi_{j,a}(A_{i,j}-\lambda)\right],\forall a=1,2,\ldots,k,
\end{align*}
where $t$ and $\lambda$ are defined in Equation (\ref{eqn:t_formula}) and Equation (\ref{eqn:lambda_formula}) respectively, and the normalization satisfies $\sum_{a=1}^k\pi'_{i,a}=1$.
\end{itemize}
\end{theorem}

All coordinate updates in Theorem \ref{thm:coor_update} have explicit formulas, which makes CAVI a computationally attractive way to approximate the global optimum $\mathbf{\hat q}^\mf$ for the community detection problem.

\subsection{Batch Coordinate Ascent Variational Inference}\label{sec:BCAVI}
The Batch Coordinate Ascent Variational Inference (BCAVI) is a batch version of CAVI. The difference lies in that CAVI updates the rows of $\pi$ sequentially one by one, while BCAVI uses the value of $\pi$ to update all rows $\{\pi'_{i,\cdot}\}$ according to Theorem \ref{thm:coor_update}. This makes BCAVI especially suitable for parallel and distributed computing, a nice feature for large scale network analysis.

We define a mapping $h:\Pi_1\rightarrow \Pi_1$ as follows. For any $\pi\in\Pi_1$, we have
\begin{align}\label{eqn:h}
[h_{t,\lambda}(\pi)]_{i,a}\propto \pipr_{i,a}\exp\left[2t\sum_{j\neq i}\pi_{ja}\left(A_{i,j}-\lambda\right)\right],
\end{align}
with parameters $t$ and $\lambda$. For BCAVI, we update $\pi$ by $\pi'=h_{t,\lambda}(\pi)$ in each batch iteration, with $t,\lambda$ defined in Equations (\ref{eqn:t_digamma}) and (\ref{eqn:lambda_digamma}). See Algorithm \ref{alg:BCAVI} for the detailed implementation of BCAVI algorithm.

\begin{algorithm}[ht]
\SetAlgoLined
\KwIn{Adjacency matrix $A$, number of communities $k$, hyperparameters $\pipr,\alphapr_p,\betapr_p,\alphapr_q,\betapr_q$, initializer $\pi^{(0)}$, number of iterations $S$.}
\KwOut{Mean variational Bayes approximation $\hat\pi,\hat\alpha_p,\hat\beta_p,\hat\alpha_q,\hat\beta_q$.}
 \For {$s=1,2,\ldots, S$}  {
 \nl Update $\alpha_p^{(s)},\beta_p^{(s)},\alpha_q^{(s)},\beta_q^{(s)}$ by 
 \begin{align}
  &\alpha^{(s)}_p = \alphapr_p + \sum_{a=1}^k\sum_{i<j}A_{i,j}\pi^{(s-1)}_{i,a}\pi^{(s-1)}_{j,a},\beta^{(s)}_p=\betapr_p+\sum_{a=1}^k\sum_{i<j}(1-A_{i,j})\pi^{(s-1)}_{i,a}\pi^{(s-1)}_{j,a},\\
  &\alpha^{(s)}_q = \alphapr_q + \sum_{a\neq b}\sum_{i<j}A_{i,j}\pi^{(s-1)}_{i,a}\pi^{(s-1)}_{j,b},\beta^{(s)}_q=\betapr_q+\sum_{a\neq b}\sum_{i<j}(1-A_{i,j})\pi^{(s-1)}_{i,a}\pi^{(s-1)}_{j,b}.
 \end{align}

\nl Define
\begin{align}
t^{(s)}&=\frac{1}{2}\left[\left[\psi(\alpha^{(s)}_p)-\psi(\beta^{(s)}_p)\right] - \left[\psi(\alpha^{(s)}_q)-\psi(\beta^{(s)}_q)\right]\right]\label{eqn:t_digamma}\\
\lambda^{(s)} &= \frac{1}{2t^{(s)}}\left[\left[\psi(\beta^{(s)}_q)-\psi(\alpha^{(s)}_q+\beta^{(s)}_q)\right]- \left[\psi(\beta^{(s)}_p)-\psi(\alpha^{(s)}_p+\beta^{(s)}_p)\right]\right],\label{eqn:lambda_digamma}
\end{align}
where $\psi(\cdot)$ is the digamma function. Then update $\pi^{(s)}$ with
\begin{align*}
\pi^{(s)}=h_{t^{(s)},\lambda^{(s)}}(\pi^{(s-1)}),
\end{align*} 
where the mapping $h(\cdot)$ is defined as in Equation (\ref{eqn:h}).
 }
 \nl We have $\hat\pi=\pi^{(S)},\hat\alpha_p=\alpha_p^{(S)},\hat\beta_p=\beta_p^{(S)},\hat\alpha_q=\alpha_q^{(S)},\hat\beta_q=\beta_q^{(S)}$.
\caption{Batch Coordinate Ascent Variational Inference (BCAVI)\label{alg:BCAVI}}
\end{algorithm}

\begin{remark}
The definitions of $t^{(s)}$ and $\lambda^{(s)}$ in Equations (\ref{eqn:t_digamma}) and (\ref{eqn:lambda_digamma}) involve the digamma function, which costs a non-negligible computational resources each time called. Note that we have $\psi(x)\in(\log(x-\frac{1}{2}),\log x)$  for all $x>1/2$. For the computational purpose, we propose to use the logarithmic function instead of digamma function in Algorithm \ref{alg:BCAVI}, i.e., Equations (\ref{eqn:t_digamma}) and (\ref{eqn:lambda_digamma}) are replaced by
\begin{align}\label{eqn:t_lambda_log}
t^{(s)}=\frac{1}{2}\log \frac{\alpha^{(s)}_p\beta^{(s)}_q}{\beta^{(s)}_p\alpha^{(s)}_q},\quad\text{and }\lambda^{(s)}=\frac{1}{2t^{(s)}}\log\frac{\beta^{(s)}_q(\alpha^{(s)}_p+\beta^{(s)}_p)}{(\alpha^{(s)}_q+\beta^{(s)}_q)\beta^{(s)}_p}.
\end{align}
Later we show that $\alpha^{(s)}_p,\beta^{(s)}_p,\alpha^{(s)}_q,\beta^{(s)}_q$ are all at least in the order of $np$, which goes to infinity, and thus the error caused by using the logarithmic function to replace the digamma function is negligible. All theoretical guarantees obtained in Section \ref{sec:main_theory} for Algorithm \ref{alg:BCAVI} (i.e., Theorem \ref{thm:mf_iterative}, Theorem \ref{thm:mf_converge}) still hold if we use Equation (\ref{eqn:t_lambda_log}) to replace Equations (\ref{eqn:t_digamma}) and (\ref{eqn:lambda_digamma}).
\end{remark}

\section{Theoretical Justifications}\label{sec:main_theory}
In this section, we establish theoretical justifications for BCAVI for community detection under the Stochastic Block Model. Though $Z$, $p$ and $q$ are all unknown, the main interest of community detection is on the recovery of the assignment matrix $Z$, while $p$ and $q$ are nuisance parameters. As a result, our main focus is on developing convergence rate of BCAVI for $\pi$.

\subsection{Loss Function}
We use $\ell_1$ norm to measure the performance of recovering $Z$. Let $\Phi$ be the set of all the bijections from $[k]$ to $[k]$. Then for any $Z,Z^*\in\Pi_1$, the loss function is defined as
\begin{align}\label{eqn:distance_ell}
\ell(Z,Z^*)&=\inf_{\phi\in\Phi}\norm{Z-\phi\circ Z^*}_{1}=\inf_{\phi\in\Phi}\sum_{i,a}|Z_{i,a}-Z_{i,\phi(a)}^*|.
\end{align}
Note that the infimum over $\Phi$ addresses the issue of identifiability over the labels. For instance, in the case of $n=4,k=2$, the assignment vector $z=(1,1,2,2)$ and $z'=(2,2,1,1)$ give the same partition. In Equation (\ref{eqn:distance_ell}) two equivalent assignments give the same loss.  

There are a few reasons for the choise of the $\ell_1$ norm. When both $Z,Z'\in\Pi_0$, the $\ell_1$ distance between $Z$ and $Z'$ is equal to the $\ell_0$ norm, i.e., the Hamming distance between the corresponding assignment vectors $r^{-1}(Z)$ and $r^{-1}(Z')$, which is the default metric used in community detection literature \cite{gao2015achieving,zhang2016minimax}. The other reason is related to the interpretation of $\Pi_1$. Since each row of $\Pi_1$ corresponds to a categorical distribution, it is natural to use the $\ell_1$ norm, the total variation distance, to measure their diffidence.

\subsection{Ground Truth}\label{sec:ground}
We use the superscript asterisk $(^*)$ to indicate the ground truth. The ground truth of connectivity matrix $B^*$ is
\begin{align*}
B^*=q^*1_k1_k^T+(p^*-q^*)I_k,
\end{align*}
where $p^*$ is the within community connection probability and $q^*$ is the between community connection probability. Throughout the paper, we assume $p^*>q^*$ such that the network satisfies the so-called ``assortative'' property, with the within-community connectivity probability larger than the between-community connectivity probability. 

We further assume the network is generated by the true assignment matrix $\Ztruth$ in the sense that $P_{i,j}= (\Ztruth B^*\Ztrutht)_{i,j}$ for all $i\neq j$. We are interested in deriving a statistical guarantee of $\ell(\hat\pi^{(s)},\Ztruth )$. Throughout this section we consider cases $\Ztruth \in\Pi_0$ or $\Ztruth \in\Pi_0^{(\rho,\rho')}$, where $\Pi_0^{(\rho,\rho')}$ is defined to be a subset of $\Pi_0$ with all the community sizes bounded between $\rho n/k$ and $\rho' n/k$. That is,
\begin{align*}
\Pi_0^{(\rho,\rho')} = \{\pi\in\Pi_0:\rho n/k \leq|\{i\in[n]:\pi_{i,a}=1\}|\leq\rho'n/k,\forall a\in[k]\}.
\end{align*}
It is worth mentioning that $\rho,\rho'$ are not necessarily constants. We allow the community sizes not to be of the same order in the theoretical analysis.

\subsection{Theoretical Justifications for BCAVI}\label{sec:theory_BCAVI}
In Theorem \ref{thm:mf_iterative}, we present theoretic guarantees of the convergence rate of BCAVI when initialized properly. Define
\begin{align*}
w=\max_{i\in[n]}\max_{a,b\in[k]}\pipr_{i,a}/\pipr_{i,b},\text{ and }\bar n_\textmin=\min_{a\neq b}[n_a+n_b]/2.
\end{align*}
When $w=1$, the priors for $\{r^{-1}(Z_{i,\cdot})\}_{i=1}^n$ are i.i.d. $\tcate(1/k,1/k,\ldots,1/k)$ and $\bar n_\textmin=n/2$ when there exist only two communities. The following quantity $I$ plays a key role in the minimax theory \cite{zhang2016minimax} 
\begin{align*}
I=-2\log \left[\sqrt{p^*q^* }+\sqrt{(1- p^*)(1- q^*)}\right],
\end{align*}
which is the R\'{e}nyi divergence of order $1/2$ between two Bernoulli distributions: $\Ber(p^*)$ and $\Ber(q^*)$. The proof of Theorem \ref{thm:mf_iterative} is deferred to Section \ref{sec:proof_mf_iterative}.

\begin{theorem}\label{thm:mf_iterative}
Let $\Ztruth \in\Pi_0$. Let $0<c_0<1$ be any constant. Assume $0<c_0p^*<q^*<p^*=o_n(1)$,
\begin{align}\label{eqn:assumption_nI}
nI/[ wk [n/\bar n_\textmin ]^2]\rightarrow\infty,\text{ and }&\alphapr_p,\betapr_p,\alphapr_q,\betapr_q=o_n((p^*-q^*)n^2/k).
\end{align}
Under the assumption that the initializer $\pi^{(0)}$ satisfies $\ell(\pi^{(0)},\Ztruth )\leq c_\tinit \bar n_\textmin$ for some sufficiently small constant $c_\tinit$ with probability  at least $1-\epsilon$, there exist some constant $c>0$ and some $\eta=o_n(1)$ such that in each iteration for the BCAVI algorithm, we have
\begin{align*}
\ell(\pi^{(s+1)},\Ztruth )\leq n\exp(-(1-\eta)\bar n_\textmin I) + \frac{\ell(\pi^{(s)},\Ztruth )}{\sqrt{nI/[wk[n/\bar n_\textmin]^2]}},\forall s\geq 0,
\end{align*}
holds uniformly with probability at least $1-\exp[-(\bar n_\textmin I)^\frac{1}{2}]-n^{-c}-\epsilon$.
\end{theorem}

Theorem \ref{thm:mf_iterative} establishes a linear convergence rate for BCAVI algorithm. The coefficient  $[nI/[wk[n/\bar n_\textmin]^2]]^{-1/2}$ is independence of $s$, and goes to 0 when $n$ grows. The following theorem is an immediate consequence of Theorem \ref{thm:mf_iterative}.


\begin{theorem}\label{thm:mf_converge}
Under the same condition as in Theorem \ref{thm:mf_iterative}, for any $s\geq s_0\triangleq [nI/k]/\log[nI/[wk[n/\bar n_\textmin]^2]]$, we have
\begin{align*}
\ell(\hat\pi^{(s)},\Ztruth )\leq n\exp(-(1-2\eta)\bar n_\textmin I)\leq\begin{cases}
n\exp(-(1-o(1))\rho nI/k),k\geq 3;\\
n\exp(-(1-o(1))nI/2),k=2,
\end{cases}
\end{align*}
with probability at least $1-\exp[-(\bar n_\textmin I)^\frac{1}{2}]-n^{-c}-\epsilon$.
\end{theorem}

Theorem \ref{thm:mf_converge} shows that BCAVI provably attains the statistical optimality from the minimax lower bound in Theorem \ref{thm:minimax} after at most $s_0$ iterations. When the network is sparse, i.e., $p^*$ and $q^*$ are at most in an order of $(\log n)/n$, the quantity $s_0$ can be shown to be $o(\log n)$, and then BCAVI converges to be minimax rate within $\log n$ iterations. When the network is dense, i.e., $p^*$ and $q^*$ are far bigger than $(\log n)/n$, $\log n$ iterations are not enough to attain the minimax rate. However, $\ell(\pi^{(s)},\Ztruth )=o(n^{-a})$ for any $a>0$ when $s\geq \log n$, and thus all the nodes can be correctly clustered with high probability by clustering each note to a community with the highest assignment probability. Therefore, it is enough to pick the number of iterations to be $\log n$ in implementing BCAVI.

\begin{theorem}\label{thm:minimax}
Under the assumption $nI/(k\log k)\rightarrow\infty$, we have 
\begin{align*}
\inf_{\hat\pi}\sup_{\Ztruth \in\Pi_0^{(\rho,\rho')}}\E \ell(\hat\pi,\Ztruth )\geq 
\begin{cases}
n\exp(-(1-o(1))\rho nI/k),k\geq 3;\\
n\exp(-(1-o(1))nI/2),k=2,
\end{cases}.
\end{align*}
\end{theorem}

Theorem \ref{thm:minimax} gives the minimax lower bound for community detection problems with respect to the $\ell(\cdot,\cdot)$ loss. In Theorem \ref{thm:mf_converge}, under the additional assumption that $\Ztruth \in\Pi_0^{(\rho,\rho')}$, it immediately reveals that BCAVI converges to the minimax rate after $s_0$ iterations. As a consequence, BCAVI is not only computationally efficient, but also achieves statistical optimality. The minimax lower bound in Theorem \ref{thm:minimax} is almost identical to the minimaxity established in \cite{zhang2016minimax}. The only difference is that \cite{zhang2016minimax} consider a $\ell_0$ loss function. The proof of Theorem \ref{thm:minimax} is just a routine extension of that in \cite{zhang2016minimax}. Therefore, we omit the proof.

To help understand Theorem \ref{thm:mf_iterative}, we add a remark on conditions on model parameters and priors, and a remark on initialization.

~\\
\noindent \textit{Remark 1 (Conditions on model parameters and priors).} The community sizes are not necessarily of the same order in Theorem \ref{thm:mf_iterative}. If we further assume $\rho,\rho'$ are constants, and the prior $\pipr_{i,a}\asymp 1/k,\forall i\in[n],a\in[k]$ (for example, uniform prior), and then the first condition in Equation (\ref{eqn:assumption_nI}) is equivalent to 
\begin{align*}
nI/k^3\rightarrow\infty,
\end{align*}
noting that $n/\bar n_\textmin \asymp k$ and $w\asymp 1$. This condition is necessary for consistent community detection \cite{zhang2016minimax} when $k$ is finite. The assumptions in Equation (\ref{eqn:assumption_nI}) is slightly stronger than the assumption in \cite{lu2016statistical}, which is essentially $nI\geq C k^2\log k$ for a sufficient large constant $C$.

Under the assumption $nI/k^3\rightarrow\infty$, since we have $I\asymp (p^*-q^*)^2/p^*$, it can be shown that $p^*,q^*$ are far bigger than $ n^{-1}$, and then the second part of Equation (\ref{eqn:assumption_nI}) can also be easily satisfied. For instance, we can simply set $\alphapr_p,\betapr_p,\alphapr_q,\betapr_q$ all equals to 1, i.e., consider non-informative priors.

~\\
\noindent \textit{Remark 2 (Initialization).} The requirement on the initializers for BCAVI in Theorem \ref{thm:mf_iterative} is relatively weak. When $k$ is a constant and the community sizes are of the same order, the condition needed is $\ell(\pi^{(0)},Z^*)\leq cn$ for some small constant $c$. Many existing methodologies in community detection literature can be used. One popular choice is spectral clustering. Established in \cite{lei2015consistency, gao2015achieving, chin2015stochastic}, the spectral clustering has a  mis-clustering error bound as $\mathcal{O}(k^2/I)$. From Equation (\ref{eqn:assumption_nI}), the error is $o(\bar n_\textmin)$, and then the condition that Theorem \ref{thm:mf_iterative} requires for initialization is satisfied. The semidefinite programming (SDP), another popular method for community detection, also enjoys satisfactory theoretical guarantees \cite{guedon2016community, fei2017exponential}, and is suitable as an initializer.

\section{Discussion}\label{sec:diss}
\subsection{Statistical Guarantee of Global Minimizer}
Though it is often challenging to obtain the global minimizer of the mean field method, it is still interesting to understand the statistical property of the global minimizer $\hat\pi^\mf$. Assume that both $p^*$ and $q^*$ are known, the optimization problem stated in Theorem \ref{thm:MF_simplifed} can be further simplified. The posterior distribution becomes $\mathbf{p}(Z|A)$. We use a product measure $\mathbf{q}_{\pi}(Z)=\prod_i \mathbf{q}_i(\pi_{i,\cdot})$ for approximation, and then $\hat\pi^\mf =\argmin_{\pi\in\Pi_1}\kl[\mathbf{q}_{\pi}(Z)\|\mathbf{p}(Z|A)]$. Theorem \ref{thm:global} reveals that $\hat\pi^\mf$ is rate-optimal, not surprisingly given the theoretical results obtained for BCAVI, an approximation of $\hat\pi^\mf$.

\begin{theorem}\label{thm:global}
Assume $p^*$ and $q^*$ are known. Under the assumption $\rho nI/[wk^2[n/\bar n_\textmin]^2]\rightarrow\infty$, there exist some constant $c>0$ and $\eta=o_n(1)$ such that 
\begin{align*}
\ell(\hat\pi^\mf ,\Ztruth )\leq n\exp(-(1-\eta)\bar n_\textmin I)
\end{align*}
with probability at least $1-\exp[-(\bar n_\textmin I)^\frac{1}{2}] -n ^{-c}$.
\end{theorem}

\subsection{Gibbs Sampling} 

In Section \ref{sec:theory_BCAVI} we analyze an iterative algorithm,  BCAVI, and establish its linear convergence towards statistical optimality. The framework and methodology we establish is not limited to BCAVI, but can be extended to other iterative algorithms, including Gibbs sampling.

As a popular Markov chain Monte Carlo (MCMC) algorithm, Gibbs sampling has been widely used in practice to approximate the posterior distribution. There is a strong tie between Gibbs sampling and the mean field variational inference: both implement coordinate updates using conditional distributions. Using the general notation introduced in Section \ref{sec:general_mf}, to approximate $\mathbf{p}(x|y)$, Gibbs sampling obtains the update on $x_i$ by a random generation from the conditional distribution $\mathbf{p}(x_i|x_{-i},y)$, while the variational inference updates in a deterministic way with $\exp\left[\E_{\mathbf{q}_{-i}} \log \mathbf{p}(x_i|x_{-i},y)\right]$. 

We present a batched version of Gibbs sampling for community detection. It involves  iterative updates with
\begin{itemize}
\item Generate $p^{(s)}$ by sampling from $\mathbf{p}(p|q^{(s-1)},Z^{(s-1)},A)$;
\item Generate $q^{(s)}$ by sampling from $\mathbf{p}(q|p^{(s-1)},Z^{(s-1)},A)$;
\item Generate $Z^{(s)}_{i,\cdot}$ independently by sampling from  $\mathbf{p}(Z^{(s-1)}_{i,\cdot}|Z^{(s-1)}_{-i,\cdot},p^{(s)},q^{(s)},A)$, for $i\in[n]$.
\end{itemize}
We include the detailed implementation as Algorithm \ref{alg:gibbs} in the supplemental material (Section \ref{sec:algo_gibbs}). The similarity between Algorithm \ref{alg:BCAVI} and Algorithm \ref{alg:gibbs} makes it possible for us to analyze the output of Gibbs sampling in a similar way as we did for the variational inference.

\begin{theorem}\label{thm:gibbs}
Assume the initializer $Z^{(0)}$ satisfies $\ell(Z^{(0)},\Ztruth )\leq c_\tinit \bar n_\textmin$ for some sufficiently small constant $c_\tinit$ with probability at least $1-\epsilon$. Under the same condition as in Theorem \ref{thm:mf_iterative}, there exist some constant $c>0$ and some $\eta,\eta'=o_n(1)$ that go to 0 slowly, such that for all $s\geq 0$ of the batched Gibbs sampling (Algorithm \ref{alg:gibbs}), we have
\begin{align*}
\E_{Z^{(s+1)}}\left[\ell(Z^{(s+1)},\Ztruth )\Big|A,Z^{(0)}\right]\leq n\exp(-(1-\eta)\bar n_\textmin I) + c_n^s\ell(Z^{(0)},\Ztruth ) + (s+1)nb_n
\end{align*}
holds with probability at least $1-\exp[-(\bar n_\textmin I)^\frac{1}{2})] -n ^{-c}  -\epsilon$, where $b_n= \exp\left[-\eta'^2 \bar n_\textmin^2\right]+ \exp\left[-\eta'^{2}n^2I\right]$ and $c_n=1/\sqrt{nI/[wk[n/\bar n_\textmin]^2]}$. Consequently, for $s = [nI/k]/\log[nI/[wk[n/\bar n_\textmin]^2]]$, we have
\begin{align*}
\E_{Z^{(s+1)}}\left[\ell(Z^{(s+1)},\Ztruth )\Big|A,Z^{(0)}\right]\leq \exp(-(1-2\eta)\bar n_\textmin I),
\end{align*}
with probability at least $1-\exp[-(\bar n_\textmin I)^\frac{1}{2}]-n^{-c}-\epsilon$.
\end{theorem}

Theorem \ref{thm:gibbs} establishes theoretical justification for batched Gibbs sampling for community detection. Despite that we have the same $c_n$ and similar convergence as Theorem \ref{thm:mf_iterative}, some extra efforts are needed due to the existence of randomness in each iterative update. The additional term of $b_n$ is necessary to handle the extreme events due to random generation. Note that $(s+1)nb_n$ is dominated by $n\exp(-(1-\eta)\bar n_\textmin I)$ as long as $s\leq e^n$. Thus, when $s\leq e^n$, we have similar ``linear convergence'' results as in Theorem \ref{thm:mf_iterative}.

\subsection{An Iterative Algorithm for Maximum Likelihood Estimation}
Maximum likelihood estimator (MLE) usually yields statistical optimality. However, the maximization of the likelihood $\mathbf{p}(A|Z,p,q)$ over $Z,p,q$ is computationally infeasible. Inspired by the procedures proposed in Algorithm \ref{alg:BCAVI} and Algorithm \ref{alg:gibbs}, we may approach $\max \mathbf{p}(A|Z,p,q)$ by alternating maximization. We use a batched coordinate maximization:
\begin{itemize}
\item Maximize $\mathbf{p}(A|p,q^{(s-1)},Z^{(s-1)})$ over $p$ to obtain $p^{(s)}$;
\item Maximize $\mathbf{p}(A|p^{(s-1)},q,Z^{(s-1)})$ over $q$ to obtain $q^{(s)}$;
\item Maximize $\mathbf{p}(A|p^{(s-1)},q^{(s-1)},Z_{i,\cdot},Z^{(s-1)}_{-i,\cdot})$ over $Z_{i,\cdot}$ to obtain $Z_{i,\cdot}^{(s)}$, for each $i\in[n]$.
\end{itemize}
We include its detailed implementation in Algorithm \ref{alg:mle} in the supplemental material (Section \ref{sec:alg_mle}). We have the following theoretical guarantee of this iterative algorithm to approximate the MLE.  

\begin{theorem}\label{thm:mle}
Assume the initializer $Z^{(0)}$ satisfies $\ell(Z^{(0)},\Ztruth )\leq c_\tinit \bar n_\textmin$ for some sufficiently small constant $c_\tinit$ with probability at least $1-\epsilon$. Under the same condition as in Theorem \ref{thm:mf_iterative}, there exist some constant $c>0$ and some $\eta=o_n(1)$, such that in each iteration of the BCAVI algorithm,
\begin{align*}
\ell(Z^{(s+1)},\Ztruth )\leq n\exp(-(1-\eta)\bar n_\textmin I) + \frac{\ell(Z^{(s)},\Ztruth )}{\sqrt{nI/[wk[n/\bar n_\textmin]^2]}},\forall s\geq 0,
\end{align*}
holds with probability at least $1-\exp[-(\bar n_\textmin I)^\frac{1}{2}]-n^{-c}-\epsilon$.
\end{theorem}

Algorithm $\ref{alg:mle}$ is essentially the same with  the procedure proposed in \cite{gao2015achieving}. However, \cite{gao2015achieving} can only analyze the performance of one single iteration from $Z^{(0)}$ (i.e., $\ell(Z^{(1)},\Ztruth )$), and it requires extra data splitting steps. Theorem \ref{thm:mle} provides a stronger and cleaner result compared with that of \cite{gao2015achieving}. 


\section{Proofs of Main Theorems}\label{sec:main_proof}
In this section, we give proofs of the theorems in Section \ref{sec:mf_all} and Section \ref{sec:main_theory}. We first present the proof of Theorem \ref{thm:MF_simplifed} in Section \ref{sec:proof_MF_simplifed}. Then we give the proof Theorem \ref{thm:coor_update} in Section \ref{sec:proof_coor_update}. The proof of Theorem \ref{thm:mf_iterative} is given in Section \ref{sec:proof_mf_iterative}.

\subsection{Proof of Theorem \ref{thm:MF_simplifed}}\label{sec:proof_MF_simplifed}
From Equation (\ref{eqn:MF}), by some algebra (see Equation (\ref{eqn:MF_simplification_proof}) in Appendix \ref{sec:appendix_general} for detailed derivation) we have 
\begin{align}\label{eqn:mf_expectation_explicit}
(\hat\pi^\mf,\hat\alpha_p^\mf,\hat\beta_p^\mf,\hat\alpha_q^\mf,\hat\beta_q^\mf) &= \argmin_{\substack{\pi\in\Pi_1 \\ \alpha_p,\beta_p,\alpha_q,\beta_q>0}} \E_{\mathbf{q}}[\log\mathbf{p}(A|Z,p,q)]-\kl(\mathbf{q}(Z,p,q)\|\mathbf{p}(Z,p,q)),
\end{align}
where we use $\mathbf{q}$ instead of $\mathbf{q}_{\pi,\alpha_p,\beta_p,\alpha_q,\beta_q}$ for simplicity. From the conditional distribution in Equation (\ref{eqn:likelihood}), the log-likelihood function can be simplified as
\begin{align*}
\log \mathbf{p}(A|Z,p,q) =\sum_{a,b}\sum_{i<j}Z_{ia}Z_{jb}\left[A_{i,j}\log\frac{B_{ab}}{1-B_{ab}}+\log(1-B_{ab})\right].
\end{align*}
Due to the independence of $Z$ and $p,q$ under $\mathbf{q}$, we have
\begin{align*}
\E_{\mathbf{q}}[\log\mathbf{p}(A|Z,p,q)]&=\E_{\mathbf{q}(p,q)}\left[\E_{\mathbf{q}(Z)}\left[\sum_{a,b}\sum_{i<j}Z_{i,a}Z_{j,b}\left[A_{i,j}\log\frac{B_{ab}}{1-B_{ab}}+\log(1-B_{ab})\right]\right]\right]\\
&=\E_{\mathbf{q}(p,q)}\left[\sum_{a,b}\sum_{i<j}\pi_{i,a}\pi_{j,b}\left[A_{i,j}\log\frac{B_{ab}}{1-B_{ab}}+\log(1-B_{ab})\right]\right].
\end{align*}
Since $B_{a,a}=p,\forall a\in[k]$ and $B_{a,b}=q,\forall a\neq b$, we have
\begin{align}\label{eqn:thm_2_1_1}
\E_{\mathbf{q}}[\log\mathbf{p}(A|Z,p,q)]&=\E_{\mathbf{q}(p,q)}\left[\sum_{a}\sum_{i<j}\pi_{i,a}\pi_{j,a}\left[A_{i,j}\log\frac{p(1-q)}{q(1-p)}+\log\frac{1-p}{1-q}\right]\right]\\
&\quad +\E_{\mathbf{q}(p,q)}\left[\sum_{a, b}\sum_{i<j}\pi_{i,a}\pi_{j,b}\Big[A_{i,j}\log\frac{q}{1-q}+\log(1-q)\Big]\right].\nonumber
\end{align}
By properties of Beta distribution, we obtain
\begin{align*}
\E_{\mathbf{q}(p,q)}\log\frac{p(1-q)}{q(1-p)} &=\E_{\mathbf{q}(p)}\left[\log p-\log(1-p)\right] - \E_{\mathbf{q}(q)}\left[\log q-\log(1-q)\right]\\
&= \left[\psi(\alpha_p)-\psi(\beta_p)\right] - \left[\psi(\alpha_q)-\psi(\beta_q)\right],
\end{align*}
and 
\begin{align*}
\E_{\mathbf{q}(p,q)}\log\frac{1-q}{1-p} &=\E_{\mathbf{q}(q)}\log(1-q) - \E_{\mathbf{q}(p)}\log(1-p)\\
&= \left[\psi(\beta_q)-\psi(\alpha_q+\beta_q)\right] - \left[\psi(\beta_p)-\psi(\alpha_p+\beta_p)\right].
\end{align*}
This leads to
\begin{align}\label{eqn:thm_2_1_2}
\E_{\mathbf{q}(p,q)} \left[\sum_{a}\sum_{i<j}\pi_{i,a}\pi_{j,a}\left[A_{i,j}\log\frac{p(1-q)}{q(1-p)}+\log\frac{1-p}{1-q}\right]\right] & = 2t\left[\sum_{a}\sum_{i<j}\pi_{i,a}\pi_{j,a}(A_{i,j}-\lambda)\right]\\
&= t\langle A-\lambda 1_n1_n^T+\lambda I_n,\pi\pi^{T}\rangle.\nonumber
\end{align}
Similarly we can obtain
\begin{align}\label{eqn:thm_2_1_3}
\E&_{\mathbf{q}(p,q)}\left[\sum_{a, b}\sum_{i<j}\pi_{i,a}\pi_{j,b}\Big[A_{i,j}\log\frac{q}{1-q}+\log(1-q)\Big]\right] \\
&= \left[\E_{\mathbf{q}(q)}\log\frac{q}{1-q}\right]\sum_{i<j}A_{i,j}\sum_{a, b}\pi_{i,a}\pi_{j,b} +\left[\E_{\mathbf{q}(q)} \log (1-q)\right]\sum_{i<j}\sum_{a, b}\pi_{i,a}\pi_{j,b}\nonumber\\
&=\frac{1}{2}\left[\psi(\alpha_q)-\psi(\beta_q)\right]\norm{A}_1 + \frac{n}{2}\left[\psi(\beta_q)-\psi(\alpha_q+\beta_q)\right],\nonumber
\end{align}
where we use the fact that $\norm{\pi_{i,\cdot}}_1=1,\forall i\in[n]$. Now consider the Kullback-–Leibler divergence between $\mathbf{q}(Z,p,q)$ and $\mathbf{p}(Z,p,q)$. Due to the independence of $p,q$ and $\{Z_{i,\cdot}\}_{i=1}^n$ in both distributions, we have
\begin{align}\label{eqn:thm_2_1_4}
\kl(\mathbf{q}(Z,p,q)\|\mathbf{p}(Z,p,q)) &= \kl(\mathbf{q}(Z)\|\mathbf{p}(Z)) + \kl(\mathbf{q}(p)\|\mathbf{p}(p)) + \kl(\mathbf{q}(q)\|\mathbf{p}(q))\\
&=\sum_{i=1}^n\kl\left[\tcate(\pi_{i,\cdot})\|\tcate(\pipr_{i,\cdot})\right] \nonumber\\
&+ \kl\left[\tbeta(\alpha_p,\beta_p)\|\tbeta(\alphapr_p,\betapr_p)\right]+ \kl\left[\tbeta(\alpha_q,\beta_q)\|\tbeta(\alphapr_q,\betapr_q)\right].\nonumber
\end{align}
By Equations (\ref{eqn:mf_expectation_explicit}) - (\ref{eqn:thm_2_1_4}), we conclude with the desired result.

\subsection{Proof of Theorem \ref{thm:coor_update}}\label{sec:proof_coor_update}
Note that 
\begin{align*}
B_{\zi,\zj} = \left[\sum_{a=1}^k Z_{i,a}Z_{j,a}\right] p+\left[\sum_{a\neq b}Z_{i,a}Z_{j,b}\right]q.
\end{align*}
We rewrite the joint distribution $\mathbf{p}(p,q,z,A)$ in Equation (\ref{eqn:joint_likelihood}) as follows,
\begin{align}\label{eqn:joint_likelihood_rewritten}
&\mathbf{p}(p,q,Z,A) \\
&= \left[\prod_{i=1}^n\pipr_{i,z_i}\right] \left[\prod_{i<j} \left[p^{A_{i,j}}(1-p)^{1-A_{i,j}}\right]^{\sum_{a=1}^k Z_{i,a}Z_{j,a}}\right] \left[\prod_{i<j} \left[q^{A_{i,j}}(1-q)^{1-A_{i,j}}\right]^{\sum_{a\neq b}^k Z_{i,a}Z_{j,b}}\right]\nonumber\\
&\quad\times\left[\frac{\Gamma(\alphapr_p+\betapr_p)}{\Gamma(\alphapr_p)\Gamma(\betapr_p)}p^{\alphapr_p-1}(1-p)^{\betapr_p-1}\right] \left[\frac{\Gamma(\alphapr_q+\betapr_q)}{\Gamma(\alphapr_q)\Gamma(\betapr_q)}q^{\alphapr_q-1}(1-q)^{\betapr_q-1}\right].\nonumber
\end{align}

\paragraph{Updates on $p$ and $q$}
From Equation (\ref{eqn:joint_likelihood_rewritten}), $p$ has conditional probability as
\begin{align*}
\mathbf{p}(p|q,Z,A)\propto\left[\prod_{i<j} \left[p^{A_{i,j}}(1-p)^{1-A_{i,j}}\right]^{\sum_{a=1}^k Z_{i,a}Z_{j,a}}\right] \left[\frac{\Gamma(\alphapr_p+\betapr_p)}{\Gamma(\alphapr_p)\Gamma(\betapr_p)}p^{\alphapr_p-1}(1-p)^{\betapr_p-1}\right].
\end{align*}
Then the CAVI update in Equation (\ref{eqn:CAVI_update_general_explicit}) leads to
\begin{align*}
 \mathbf{\hat q}(p)&\propto \exp\left[\E_{\mathbf{q}(q,Z)} \log \mathbf{p}(p|q,Z,A)\right]\\
&\propto    \exp\left[ \E_{\mathbf{q}(Z)}\sum_{i<j}\sum_{a=1}^k Z_{i,a}Z_{j,a}\log\left[p^{A_{i,j}}(1-p)^{1-A_{i,j}}\right] \right] \left[\frac{\Gamma(\alphapr_p+\betapr_p)}{\Gamma(\alphapr_p)\Gamma(\betapr_p)}p^{\alphapr_p-1}(1-p)^{\betapr_p-1}\right]\\
&=    \exp\left[ \sum_{i<j}\sum_{a=1}^k \pi_{i,a}\pi_{j,a}\log\left[p^{A_{i,j}}(1-p)^{1-A_{i,j}}\right] \right] \left[\frac{\Gamma(\alphapr_p+\betapr_p)}{\Gamma(\alphapr_p)\Gamma(\betapr_p)}p^{\alphapr_p-1}(1-p)^{\betapr_p-1}\right].
\end{align*}
It can be written as
\begin{align*}
 \mathbf{\hat q}(p)&\propto \left[p^{\sum_{i<j}\sum_{a=1}^k \pi_{i,a}\pi_{j,a}A_{i,j}}(1-p)^{\sum_{i<j}\sum_{a=1}^k \pi_{i,a}\pi_{j,a}(1-A_{i,j})}\right]  \left[\frac{\Gamma(\alphapr_p+\betapr_p)}{\Gamma(\alphapr_p)\Gamma(\betapr_p)}p^{\alphapr_p-1}(1-p)^{\betapr_p-1}\right].
\end{align*}
The distribution of $p$ is still Beta $p\sim \tbeta (\alpha'_p,\beta'_p)$, with
\begin{align*}
\alpha'_p = \alphapr_p + \sum_{i<j}\sum_{a=1}^k \pi_{i,a}\pi_{j,a}A_{i,j},\text{ and }\beta_p'=\betapr_p+\sum_{i<j}\sum_{a=1}^k \pi_{i,a}\pi_{j,a}(1-A_{i,j}).
\end{align*}
Similar analysis on $q$ yields updates on $\alpha'_q$ and $\beta'_q$. Hence, its proof is omitted.

\paragraph{Updates on $\{Z_{i,\cdot}\}_{i=1}^n$} From Equation (\ref{eqn:joint_likelihood_rewritten}), the conditional distribution on $Z_{i,\cdot}$ is
\begin{align*}
\mathbf{p}(Z_{i,\cdot}|Z_{-i,\cdot},p,q,A)\propto \pipr_{i,z_i}\left[\prod_{j\neq i}B_{\zi,\zj}^{A_{i,j}}(1-B_{\zi,\zj})^{1-A_{i,j}}\right].
\end{align*}
Consequently, up to a constant not depending on $i$, we have
\begin{align*}
&\log \p(Z_{i,a} = 1|Z_{-i,\cdot},p,q,A) \\
&= \log \pipr_{i,a} +\log\left[\sum_{j\neq i}Z_{j,a}\left[A_{i,j}\log \frac{p}{1-p} +\log(1-p)\right]+\sum_{j\neq i}\sum_{b\neq a}Z_{j,b}\left[A_{i,j}\log \frac{q}{1-q} +\log(1-q)\right]\right]\\
&= \log \pipr_{i,a} +\log\left[\sum_{j\neq i}Z_{j,a}\left[A_{i,j}\log \frac{p(1-q)}{q(1-p)} -\log\frac{1-q}{1-p}\right]+\sum_{j\neq i}\left[A_{i,j}\log \frac{q}{1-q} +\log(1-q)\right]\right].
\end{align*}
Then the CAVI update from Equation (\ref{eqn:CAVI_update_general_explicit}) leads to
\begin{align}
\pi_{i,a}' &=  \mathbf{\hat q}_{Z_{i,\cdot}}(Z_{i,a}=1)\nonumber\\
&\propto\exp\left[\E_{\mathbf{q}(p,q,z_{-i})}\log \p(Z_{i,a} = 1|Z_{-i,\cdot},p,q,A)\right]\nonumber\\
&= \exp\left[\E_{\mathbf{q}(p)}\E_{\mathbf{q}(q)}\E_{\mathbf{q}(Z_{-i,\cdot})}\log \p(Z_{i,}=1 |Z_{-i,\cdot},p,q,A)\right]\nonumber\\
&\propto\pipr_{i,a}\exp\left[\E_{\mathbf{q}(p)}\E_{\mathbf{q}(q)}\sum_{j\neq i}\pi_{j,a}\left[A_{i,j}\log \frac{p(1-q)}{q(1-p)} -\log\frac{1-q}{1-p}\right]\right],\label{eqn:pi_update}
\end{align}
where we use the property that $p,q,Z$ are all independent of each other under $\mathbf{q}$. Recall that $p\sim\tbeta(\alpha_p,\beta_p)$ and $q\sim\tbeta(\alpha_q,\beta_q)$. It can be shown that
\begin{align*}
\E_{\mathbf{q}(p)}\log\frac{p}{1-p}=\psi(\alpha_p)-\psi(\beta_p),\text{ and }\E_{\mathbf{q}(p)}\log(1-p)=\psi(\beta_p)-\psi(\alpha_p+\beta_p),
\end{align*}
where $\psi(\cdot)$ is digamma function. Similar results hold for $\E_{\mathbf{q}(q)}\log (q/(1-q))$ and $\E_{\mathbf{q}(q)}\log(1-q)$. Plug in these expectations to Equation (\ref{eqn:pi_update}), we have
\begin{align*}
\pi_{i,a}' &\propto\pipr_{i,a}\exp\left[2t\sum_{j\neq i}\pi_{j,a}(A_{i,j}-\lambda)\right].
\end{align*}

\subsection{Proof of Theorem \ref{thm:mf_iterative}}\label{sec:proof_mf_iterative}

Theorem \ref{thm:mf_iterative} gives a theoretical justification for all iterations in the BCAVI algorithm. Due to the limit of pages, in this section we assume $\ell(\pi^{(0)},Z^*)=o(\bar n_\textmin)$. The proof of the case $\ell(\pi^{(0)},Z^*)$ in a constant order of $\bar n_\textmin$ is essentially the same with slight modification, and we defer it to Section \ref{sec:proof_mf_iterative_constant} in the supplemental material.

To prove the theorem, it is sufficient if we are able to show the loss $\ell(\cdot,\Ztruth )$ decreases in a desired way for one BCAVI iteration, when the community assignment is in an appropriate neighborhood of the truth.  Let $\gamma=o(1)$ be any sequence that goes to zero when $n$ grows. Define $t^*$ and $\lambda^*$ as the true counterparts of $t$ and $\lambda$, by
\begin{align*}
t^*=\frac{1}{2}\log\frac{p^*(1-q^*)}{q^*(1-p^*)},\text{ and }\lambda^* =\frac{1}{2t^*}\log \frac{1-q^*}{1-p^*}.
\end{align*}
The proof of Theorem \ref{thm:mf_iterative} involves  three parts as follows.

~\\
\textbf{Part One: One Iteration.} Consider any $\pi\in\Pi_1$ such that $\norm{\pi-\Ztruth }_1\leq \gamma\bar n_\textmin$. Let $\eta'$ be any sequence such that  $\eta'=o(1)$. Consider any $t$ and $\lambda$ with $|t-t^*|\leq \eta' (p^*-q^*)/p^*$ and $|\lambda-\lambda^*|\leq \eta' (p^*-q^*)$.  We define $\mathcal{F}$ to be the event, that after applying the mapping $h_{t,\lambda}(\cdot)$, there exists some $\eta=o(1)$ such that
\begin{align*}
\norm{h_{t,\lambda}(\pi)-\Ztruth }_1 \leq n\exp(-(1-\eta)\bar n_\textmin I) + \frac{\norm{\pi-\Ztruth }_1}{\sqrt{nI/[wk[n/\bar n_\textmin]^2]}},
\end{align*}
holds uniformly over all the eligible $\pi,t$ and $\lambda$. We have
\begin{align*}
\p(\mathcal{F})\geq 1-\exp[-(\bar n_\textmin I)^\frac{1}{2})] -n ^{-r},
\end{align*}
for some constant $r>0$. We defer its proof to the later part of this section.

~\\
\textbf{Part Two: Consistency of Model Parameters.} Consider any $\pi\in\Pi_1$ such that $\norm{\pi-\Ztruth }_1\leq \gamma\bar n_\textmin$. Define
\begin{align}
  &\alpha_p = \alphapr_p + \sum_{a=1}^k\sum_{i<j}A_{i,j}\pi_{i,a}\pi_{j,a},\quad\beta_p=\betapr_p+\sum_{a=1}^k\sum_{i<j}(1-A_{i,j})\pi_{i,a}\pi_{j,a},\label{eqn:alpha_p_prime}
\end{align}
and
\begin{align}
  &\alpha_q = \alphapr_q + \sum_{a\neq b}\sum_{i<j}A_{i,j}\pi_{i,a}\pi_{j,b},\quad\beta_q=\betapr_q+\sum_{a\neq b}\sum_{i<j}(1-A_{i,j})\pi_{i,a}\pi_{j,b},\label{eqn:alpha_q_prime}
 \end{align}
 and consequently,
 \begin{align}
 t&=\frac{1}{2}\left[\left[\psi(\alpha_p)-\psi(\beta_p)\right] - \left[\psi(\alpha_q)-\psi(\beta_q)\right]\right]\label{eqn:t_prime}\\
\lambda &= \frac{1}{2t}\left[\left[\psi(\beta_q)-\psi(\alpha_q+\beta_q)\right]- \left[\psi(\beta_p)-\psi(\alpha_p+\beta_p)\right]\right].\label{eqn:lambda_prime}
 \end{align}
From Lemma \ref{lem:concentration_t_lambda}, we have a concentration of $t,\lambda$ towards $t^*,\lambda^*$. That is, there exists some $\eta' =o(1)$, such that with probability at least $1-e^35^{-n}$, the following inequalities hold
 \begin{align*}
 |t-t^*|\leq \eta' (p^*-q^*)/p^*,\text{ and }|\lambda-\lambda^*|\leq \eta' (p^*-q^*),
 \end{align*}
 uniformly over all the eligible $\pi$.
 
 ~\\
\textbf{Part Three: Multiple Iterations.} Consider any $\pi\in\Pi_1$ such that $\norm{\pi-\Ztruth }_1\leq \gamma\bar n_\textmin$. Define $\alpha_p,\beta_p,\alpha_q,\beta_q,t,\lambda$ as Equations (\ref{eqn:alpha_p_prime}) - (\ref{eqn:lambda_prime}). A combination of results from \emph{Part One} and \emph{Part Two} immediately implies that
\begin{align}\label{eqn:convergence_pi_prime}
\norm{h_{t,\lambda}(\pi)-\Ztruth }_1\leq n\exp(-(1-\eta)\bar n_\textmin I) + \frac{\norm{\pi-\Ztruth }_1}{\sqrt{nI/[wk[n/\bar n_\textmin]^2]}},
\end{align}
holds uniformly over all the eligible $\pi$ with probability at least $1-\exp[-(\bar n_\textmin I)^\frac{1}{2})]-n ^{-r}$. This is sufficient to show Theorem \ref{thm:mf_iterative}.

 ~\\
\indent The only thing left to be proved, the most critical part towards the proof of Theorem \ref{thm:mf_iterative}, is the claim we made in \emph{Part One}. We are going to prove the claim as follow.

 ~\\
\textbf{Proof Sketch of Part One.}
The error associated with the $[h_{t,\lambda}(\pi)]_{i,\cdot}$ is a function of $\pi$ and $A_{i,\cdot}$. It can be decomposed into a summation of two terms, one only involves the ground truth $\Ztruth $ and the other involves  the deviation $\pi-\Ztruth $. That is,
\begin{align*}
\norm{[h_{t,\lambda}( \pi)]_{i,\cdot}-\Ztruth _{i,\cdot}}_1 \leq f_{i,1}(\Ztruth ,A_{i,\cdot}) + f_{i,2}(\pi-\Ztruth ,A_{i,\cdot}).
\end{align*}
Consequently,
\begin{align}\label{eqn:proof_step_one_sketch}
\norm{h_{t,\lambda}(\pi)-\Ztruth }_1 \leq \underbrace{\sum_{i=1}^nf_{i,1}(\Ztruth ,A_{i,\cdot})}_{\text{involves }\Ztruth } + \underbrace{\sum_{i=1}^nf_{i,2}(\pi-\Ztruth ,A_{i,\cdot})}_{\text{involves }\pi-\Ztruth }.
\end{align}
With a proper choice of $f_{\cdot,1}$ and $f_{\cdot,2}$, the first term on the RHS of Equation (\ref{eqn:proof_step_one_sketch}) leads to the minimax rate $n\exp(-(1-\eta)\bar n_\textmin I) $. Up to a constant not dependent on $\pi,\Ztruth $ or $A$, the second term can be written as
\begin{align*}
\sum_{i=1}^nf_{i,2}(\pi-\Ztruth ,A_{i,\cdot}) \lesssim \sum_a (\pi_{\cdot,a}-\Ztruth _{\cdot,a})^T(A-\E A)(A-\E A)^T(\pi_{\cdot,a}-\Ztruth _{\cdot,a}).
\end{align*}
In this way it is all about the random matrix $A-\E A$ and there exist sharp bounds on $\opnorm{A-\E A}$. Note that $\sum_a\norm{\pi_{\cdot,a}-\Ztruth_{\cdot,a}}^2\leq \sum_a\norm{\pi_{\cdot,a}-\Ztruth_{\cdot,a}}_1\leq \norm{\pi-\Ztruth}_1$. The second term ends up being upper bounded by $\norm{\pi-\pi^*}_1$ multiplied by a coefficient factor.

 ~\\
\textbf{Proof of Part One. }
Denote $z=r^{-1}(\Ztruth)$. By the definition of $h_{t,\lambda}(\cdot)$ in Equation (\ref{eqn:h}), we have
\begin{align*}
\norm{[h_{t,\lambda}( \pi)]_{i,\cdot}-\Ztruth_{i,\cdot}}_1&\leq \frac{2\sum_{a\neq \zi}\pipr_{i,a}\exp\left[2t\sum_{j\neq i}\pi_{j,a}(A_{i,j}-\lambda)\right]}{\sum_{a}\pipr_{i,a}\exp\left[2t\sum_{j\neq i}\pi_{j,a}(A_{i,j}-\lambda)\right]}\\
&\leq 2w\sum_{a\neq \zi}1\wedge\exp\left[2t\sum_{j\neq i}(\pi_{j,a}-\pi_{j,\zi})(A_{i,j}-\lambda)\right].
\end{align*}
Define $f(x)=1\wedge \exp(-x)$. It can be shown that for any $x_0<0$ and any integer $m\geq 1$ we have $f(x)\leq \exp(x_0)+\sum_{l=0}^{m-1}\exp(l x_0/m)\mathbb{I}\{x\geq (l+1) x_0/m\}$, which can be seen as a stepwise approximation of the continuous function $f(x)$. By taking $x_0=-(n_a+n_{\zi})I/2$ and letting $x=2t\sum_{j\neq i}(\pi_{j,a}-\pi_{j,\zi})(A_{i,j}-\lambda)$, we have
\begin{align*}
\norm{[h_{t,\lambda}( \pi)]_{i,\cdot}-\Ztruth_{i,\cdot}}_1&\leq 2w\sum_{a\neq z_i}\exp\left[-\frac{(n_a+n_{\zi})I}{2}\right]+2w\sum_{l=0}^{m-1}\Bigg[\exp\left[-\frac{l(n_a+n_{\zi})I}{2m}\right]\\
&\quad \times \sum_{a\neq \zi}\mathbb{I}\bigg[2t\sum_{j\neq i}(\pi_{j,a}-\pi_{j,\zi})(A_{i,j}-\lambda)\geq -\frac{(l+1)(n_a+n_{\zi})I}{2m}\bigg]\Bigg].
\end{align*}
We choose some $m\rightarrow\infty$ slowly such that 
\begin{align}\label{eqn:m}
m=o(\bar n_\textmin I)\text{ and }m=o([wnI/[k[n/\bar n_\textmin]^2]^{1/4}).
\end{align}
Thus, we have
\begin{align}
\norm{h_{t,\lambda}(\pi)-\Ztruth}_1&\leq 2wnk\exp(-\bar n_\textmin I)+2w\sum_{l=0}^{m-1}\sum_{a=1}^k\sum_{b\neq a}\Bigg[\exp\left[-\frac{l(n_a+n_b)I}{2m}\right]\nonumber\\
&\quad\times \sum_{i:\zi=b}\mathbb{I}\bigg[\sum_{j\neq i}(\pi_{j,a}-\pi_{j,b})(A_{i,j}-\lambda)\geq -\frac{(l+1)(n_a+n_b)I}{4mt}\bigg]\Bigg]\label{eqn:h_pi_l1}
\end{align}
where we use the fact that $\min_{a\neq b}(n_a+b_b)/2\geq \bar n_\textmin$. 

The key to the rest of the analysis is to understand Equation (\ref{eqn:h_pi_l1}) through the decomposition of the critical quantity $\sum_{j\neq i}(\pi_{j,a}-\pi_{j,b})(A_{i,j}-\lambda)$. We will show for any pair of $a,b\in[k]$ such that $a\neq b$, and any $i\in[n]$ such that $\zi=b$, it is equal to a summation of two terms: one only involves the ground truth $\Ztruth$, and the other involves the deviation $\pi-\Ztruth$. The former remains steady along iterations and contributes to the minimax rate, while the latter needs to be connected with the error $\norm{\pi-\Ztruth}_1$.

Let $\theta_{a,b}$ be a vector of length $n$ such that $[\theta_{a,b}]_j = \pi_{j,a} - \Ztruth_{j,a}+\Ztruth_{j,b}-\pi_{j,b},\forall j\in[n]$. Then we have
\begin{align}\label{eqn:pi_decompose_rewrite}
\sum_{j\neq i}(\pi_{j,a}-\pi_{j,b})(A_{i,j}-\lambda) &= \sum_{j\neq i}(\Ztruth_{j,a}-\Ztruth_{j,b})(A_{i,j}-\lambda)+\sum_{j\neq i}(\pi_{j,a} - \Ztruth_{j,a}+\Ztruth_{j,b}-\pi_{j,b})(A_{i,j}-\lambda)\\
&=\sum_{j\neq i}(\Ztruth_{j,a}-\Ztruth_{j,b})(A_{i,j}-\lambda)+\sum_{j\neq i} (A_{i,j}-\lambda)[\theta_{a,b}]_j\nonumber\\
&=\underbrace{\sum_{j\neq i}(\Ztruth_{j,a}-\Ztruth_{j,b})(A_{i,j}-\lambda)}_{\text{involves }\Ztruth} + \underbrace{(A_{i,\cdot}-\E A_{i,\cdot})\theta_{a,b} +\sum_{j\neq i}(\E A_{i,j}-\lambda)[\theta_{a,b}]_j}_{\text{involves }\pi-\Ztruth}.\nonumber
\end{align}

With the help of Equation (\ref{eqn:pi_decompose_rewrite}), Equation (\ref{eqn:h_pi_l1}) can be written as
\begin{align*}
&\norm{h_{t,\lambda}(\pi)-\Ztruth}_1\\
&\leq 2wnk\exp(-\bar n_\textmin I)+2w\sum_{l=0}^{m-1}\sum_{a=1}^k\sum_{b\neq a}\Bigg[\exp\left[-\frac{l(n_a+n_b)I}{2m}\right]\\
&\quad\times \sum_{i:\zi=b}\mathbb{I}\bigg[\sum_{j\neq i}(\Ztruth_{j,a}-\Ztruth_{j,b})(A_{i,j}-\lambda)\geq -\frac{(l+3/2)(n_a+n_b)I}{4mt}-\sum_{j\neq i}(\E A_{i,j}-\lambda)[\theta_{a,b}]_j\bigg]\Bigg]\\
&+2w\sum_{a=1}^k\sum_{b\neq a}\Bigg[\Bigg[\sum_{l=0}^{m-1}\exp\left[-\frac{l(n_a+n_b)I}{2m}\right]\Bigg]\times\sum_{i:\zi=b}\mathbb{I}\left[(A_{i,\cdot}-\E A_{i,\cdot})\theta_{a,b}\geq \frac{\bar n_\textmin I}{4mt}\right]\Bigg].
\end{align*}
Equations (\ref{eqn:assumption_nI}) and (\ref{eqn:m}) imply $\sum_{l=0}^{m-1}\exp\left[-l(n_a+n_b)I/(2m)\right]\leq 2$. Thus, we have
\begin{align*}
\norm{h_{t,\lambda}(\pi)-\Ztruth}_1\leq 2wnk\exp(-\bar n_\textmin I)+\underbrace{2wL_1^\textsum }_{\text{involves } \Ztruth}+ \underbrace{4wL_2^\textsum}_{\text{involves  }\pi-\Ztruth},
\end{align*}
where
\begin{align*}
L_1^\textsum\triangleq\sum_{l=0}^{m-1}\sum_{a=1}^k\sum_{b\neq a}\exp\left[-\frac{l(n_a+n_b)I}{2m}\right]\sum_{i:\zi=b}L_{1,i}(a,b,l),
\end{align*}
with $L_{1,i}(a,b,l)\triangleq\mathbb{I}[\sum_{j\neq i}(\Ztruth_{j,a}-\Ztruth_{j,b})(A_{i,j}-\lambda)\geq -(l+3/2)(n_a+n_b)I/(4mt)-\sum_{j\neq i}(\E A_{i,j}-\lambda)[\theta_{a,b}]_j]$, and 
\begin{align*}
L_2^\textsum\triangleq \sum_{a=1}^k\sum_{b\neq a}\sum_{i:\zi=b}\mathbb{I}\left[(A_{i,\cdot}-\E A_{i,\cdot})\theta_{a,b}\geq \frac{\bar n_\textmin I}{4mt}\right].
\end{align*}
In this way we turn $\norm{h_{t,\lambda}(\pi)-\Ztruth}_1$ into calculations on $L_1^\textsum$ and $L_2^\textsum$, where the former only involves  the ground truth $\Ztruth$ and the latter only involves  the deviation $\pi-\Ztruth$.

We can obtain upper bounds on $L_{1}^\text{sum}$ and $L_{2}^\text{sum}$ as follows. Their proofs are deferred to the end of this section.
\begin{itemize}
\item For $L_{1}^\text{sum}$, there exists a sequence $\eta''=o(1)$ such that with probability at least $1-\exp[-2(\bar n_\textmin I)^\frac{1}{2}]$, we have
\begin{align}\label{eqn:upper_l1sum}
L_1^\textsum \leq nmk\exp\left[-(1-2\eta'')\bar n_\textmin I\right].
\end{align}
\item For $L_{2}^\text{sum}$, there exist constants $c$ and $r$ such that with probability at least $1-n^{-r} -\exp(-5np^*)$, we have
\begin{align}\label{eqn:upper_l2sum}
L_{2}^\textsum\leq \frac{cknp^* \norm{\pi-\Ztruth}_1}{(\bar n_\textmin I/(mt^*))^2} +\frac{cn^2kp^*\exp(-5np^*)}{\bar n_\textmin I/(mt^*)}.
\end{align}
\end{itemize}
Thus, we have
\begin{align*}
\norm{h_{t,\lambda}(\pi)-\Ztruth}_1&\leq 2wnk\exp(-\bar n_\textmin I)+2wnmk\exp\left[-(1-2\eta'')\bar n_\textmin I\right]\\
&\quad  +\frac{4cwknp^* \norm{\pi-\Ztruth}_1}{(\bar n_\textmin I/(mt^*))^2} +\frac{4cwkn^2p^*\exp(-5np^*)}{\bar n_\textmin I/(mt^*)},
\end{align*}
with probability at least $1-\exp[-2(\bar n_\textmin I)^\frac{1}{2}]-n^{-r}-\exp(-5np^*)$. By Propositions \ref{prop:I} and \ref{prop:lambda}, we have $p^*t^{*2}\asymp I$. Then due to Equation (\ref{eqn:m}), we have
\begin{align*}
\frac{w knp^* }{(\bar n_\textmin I/(mt^*))^2}\asymp wm^2\left[\frac{n}{\bar n_\textmin}\right]^2\frac{k}{nI}=o\left[\frac{1}{\sqrt{nI/[wk[n/\bar n_\textmin]^2]}}\right],
\end{align*}
and
\begin{align*}
\frac{wkn^2p^*\exp(-5np^*)}{\bar n_\textmin I/(mt^*)}\asymp wmk\frac{\sqrt{np^*}}{\sqrt{nI}}\left[\frac{n}{\bar n_\textmin}\right]n\exp(-5np^*)\leq n\exp(-5\bar n_\textmin I).
\end{align*}
Thus, with probability at least $1-\exp[-(\bar n_\textmin I)^\frac{1}{2}]-n^{-r}$, there exists some $\eta=o(1)$, such that
\begin{align*}
\norm{h_{t,\lambda}(\pi)-\Ztruth}_1\leq n\exp(-(1-\eta)\bar n_\textmin I) + \frac{\norm{\pi-\Ztruth}_1}{\sqrt{nI/[wk[n/\bar n_\textmin]^2]}}.
\end{align*}

~\\
\indent The proof for \emph{Part One} is complete. The very last thing remained to be obtained is upper bounds on $L_{1}^\text{sum}$ and $L_{2}^\text{sum}$, i.e., Equations (\ref{eqn:upper_l1sum}) and (\ref{eqn:upper_l2sum}). Recall the definition of $\theta_{a,b}$. We have some properties on $\theta_{a,b}$ which will be useful in the analysis for $L_{1}^\text{sum}$ and $L_{2}^\text{sum}$:  $\norm{\theta_{a,b}}_\infty\leq 2$ and
\begin{align}\label{eqn:theta_ab}
\norm{\theta_{a,b}}_1\leq \norm{\pi_{\cdot,a}-\Ztruth_{\cdot,a}}_1+\norm{\pi_{\cdot,b}-\Ztruth_{\cdot,b}}_1\leq \norm{\pi-\Ztruth}_1\leq \gamma \bar n_\textmin,
\end{align}
and
\begin{align}\label{eqn:theta_sum}
\sum_{a=1}^k\sum_{b\neq a}\norm{\theta_{a,b}}_1\leq 2k\sum_a\norm{\pi_{\cdot,a}-\Ztruth_{\cdot,a}}_1 \leq 2k\norm{\pi-\Ztruth}_1.
\end{align}
~\\
\noindent\textit{1. Bounds on $L_{1}^\text{sum}$.} \quad By applying Markov inequality, we have
\begin{align*}
&\E L_{1,i}(a,b,l)\\
&=\p\left[t^*\sum_{j\neq i}(\Ztruth_{j,a}-\Ztruth_{j,b})(A_{i,j}-\lambda)\geq -\frac{t^*(l+3/2)(n_a+n_b)I}{4mt}-t^*\sum_{j\neq i}(\E A_{i,j}-\lambda)[\theta_{a,b}]_j\right]\\
&\leq \exp\left[\frac{t^*(l+3/2)(n_a+n_b)I}{4mt}+t^*(\E A_{i,j}-\lambda 1_n^T)\theta_{a,b}\right] \E\exp\left[t^*\sum_{j\neq i}(\Ztruth_{j,a}-\Ztruth_{j,b})(A_{i,j}-\lambda)\right].
\end{align*}
With the help of Proposition \ref{prop:chernoff_I}, we have
\begin{align*}
&\E\exp\left[t^*\sum_{j\neq i}(\Ztruth_{j,a}-\Ztruth_{j,b})(A_{i,j}-\lambda)\right]\\
&= \exp(-t^*(\lambda-\lambda^*)(n_a-n_b))\exp(-t^*\lambda^*(n_a-n_b))\prod_{j\neq i}\E\exp(t^*(\Ztruth_{j,a}-\Ztruth_{j,b})A_{i,j})\\
&= \exp(-t^*(\lambda-\lambda^*)(n_a-n_b))\left[e^{-t\lambda}\frac{\E e^{tX}}{\E e^{-tY}}\right]^\frac{n_a-n_b}{2}\left[\E e^{tX}\E e^{-tY}\right]^\frac{n_a+n_b}{2}\\
&= \exp(-t^*(\lambda-\lambda^*)(n_a-n_b))\exp\left[-\frac{(n_a+n_b)I}{2}\right].
\end{align*}
Hence
\begin{align}\label{eqn:expect_L_1}
&\E L_1^\textsum \\
&= \sum_{l=0}^{m-1}\sum_{a=1}^k\sum_{b\neq a}\Bigg[\exp\left[-\frac{l(n_a+n_b)I}{2m}\right] \exp\left[\frac{t^*(l+3/2)(n_a+n_b)I}{4mt}+t^*\sum_{j\neq i}(\E A_{i,j}-\lambda)[\theta_{a,b}]_j\right]\nonumber\\
&\quad\times\exp(-t^*(\lambda-\lambda^*)(n_a-n_b))\exp\left[-\frac{(n_a+n_b)I}{2}\right]\Bigg]\nonumber\\
&\leq \sum_{l=0}^{m-1}\sum_{a=1}^k\sum_{b\neq a}\exp\left[-\frac{(1+\frac{l}{m}-\frac{t^*(l+3/2)}{2mt})(n_a+n_b)I}{2}-t^*(\lambda-\lambda^*)(n_a-n_b)+t^*\sum_{j\neq i}(\E A_{i,j}-\lambda)[\theta_{a,b}]_j\right].\nonumber
\end{align}
We are going to show $-(1-\eta'')\bar n_\textmin I$ upper bounds terms in the exponent of RHS of Equation (\ref{eqn:expect_L_1}) by some $\eta''=o(1)$. We first present some properties of $\lambda^*,t^*$ and $I$ that will be helpful:
\begin{align}
& I \asymp (p^*-q^*)^2/p^*,\label{eqn:I_t}\\
&\lambda^*\in(q^*,p^*) \label{eqn:lambda_interval}, \\
\text{and }& t^*\asymp (p^*-q^*)/p^*. \label{eqn:t_star_simple}
\end{align}
Here Equations (\ref{eqn:I_t}) and (\ref{eqn:lambda_interval})  are proved by Propositions \ref{prop:I}  and \ref{prop:lambda}  respectively. Equation (\ref{eqn:t_star_simple}) is due to $t^*\asymp \log(1+(p^*-q^*)/q^*)\asymp (p^*-q^*)/p^*$ under the assumption that $p^*,q^*=o(1)$, $p^*\asymp q^*$.

The first term in the exponent of Equation (\ref{eqn:expect_L_1}) is upper bounded by $-(1-7/(8m))\bar n_\textmin I$ by the assumption $t^*/t=1+o(1)$. Since $|t^*(\lambda -\lambda^*)|\leq \eta't^*(p^*-q^*)$, by Equations (\ref{eqn:I_t}) and (\ref{eqn:t_star_simple}) the second term is upper bounded by $\eta'\bar n_\textmin I$ up to a constant factor. For the last term in the exponent of Equation (\ref{eqn:expect_L_1}), since $|\lambda-\lambda^*|\leq \eta'(p^*-q^*)$ we have
\begin{align*}
t^*\bigg|\sum_{j\neq i}(\E A_{i,j}-\lambda)[\theta_{a,b}]_i\bigg| & \leq t^*\bigg|\sum_{j\neq i}(\E A_{i,j}-\lambda^*)[\theta_{a,b}]_i\bigg| + t^*\bigg|\sum_{j\neq i}(\lambda^*-\lambda)[\theta_{a,b}]_i\bigg|\\
&\leq (1+\eta') t^*(p^*-q^*)\norm{\theta_{a,b}}_1\\
&\leq (1+\eta') t^*(p^*-q^*) \gamma\bar n_\textmin\\
&\lesssim \gamma \bar n_\textmin I,
\end{align*}
where we use Equations (\ref{eqn:theta_ab}) and (\ref{eqn:I_t}) - (\ref{eqn:t_star_simple}).

As a consequence, there exists a sequence $\eta''=o(1)$ that goes to zero slower than $m^{-1}, \gamma, \eta'$, such that the summation of three terms in the exponent of the RHS of Equation (\ref{eqn:expect_L_1}) is upper bounded by $-(1-\eta'')\bar n_\textmin I$. Thus, Equation (\ref{eqn:expect_L_1}) can be written as
\begin{align*}
\E L_1^\textsum \leq nmk\exp\left[-(1-\eta'')\bar n_\textmin I\right].
\end{align*}
Since $\eta''$ goes to 0 slower than $m^{-1}$, we have $\eta ''\geq m^{-1}\geq (\bar n_\textmin I)^\frac{1}{4}$ by Equation (\ref{eqn:m}). Then by applying Markov inequality, we have
\begin{align*}
\p\left[L_1^\textsum \geq nmk\exp\left[-(1-2\eta'')\bar n_\textmin I\right]\right] \leq \exp\left[-\eta''\bar n_\textmin I\right]\leq \exp\left[-2(\bar n_\textmin I)^\frac{1}{2}\right].
\end{align*}
That is, with probability at least $1-\exp[-2(\bar n_\textmin I)^\frac{1}{2}]$, Equation (\ref{eqn:upper_l1sum}) holds.

~\\
\noindent\textit{2. Bounds on $L_{2}^\textsum$.} Depending on whether the network is dense or sparse, we consider two scenarios.

\textit{(1) Dense Scenario: $q^*\geq (\log n)/n$.} In this scenario, we have a sharp bound on $\opnorm{A-\E A}$. First we observe that
\begin{align*}
\sum_{i:\zi=b}[(A_{i,\cdot}-\E A_{i,\cdot})\theta_{a,b}]^2& =\theta_{a,b}^T \sum_{i:\zi=b}[(A_{i,\cdot}-\E A_{i,\cdot})^T(A_{i,\cdot}-\E A_{i,\cdot})]\theta_{a,b}\\
&\leq \theta_{a,b}^T \sum_i[(A_{i,\cdot}-\E A_{i,\cdot})^T(A_{i,\cdot}-\E A_{i,\cdot})]\theta_{a,b}\\
&= \theta_{a,b}^T [(A-\E A)^T(A-\E A)]\theta_{a,b}.
\end{align*}
By applying Markov inequality, we have 
\begin{align*}
L_2^\text{sum}& \leq \sum_{a=1}^k\sum_{b\neq a} \frac{\theta_{a,b}^T [(A-\E A)^T(A-\E A)]\theta_{a,b}}{(\bar n_\textmin I/(4mt))^2}.
\end{align*}
Since $\norm{\theta_{a,b}}_\infty\leq 2$, we have $\norm{\theta_{a,b}}^2\leq 2\norm{\theta_{a,b}}_1$. Lemma \ref{lem:A_opnorm} shows $\opnorm{A-\E A}\leq \sqrt{c_1np}$ holds with probability at least $1-n^{-r}$ for some constants $c_1,r>0$. Together with  Equation (\ref{eqn:theta_sum}), we have
\begin{align*}
\sum_{a=1}^k\sum_{b\neq a}\theta_{a,b}^T [(A-\E A)^T(A-\E A)]\theta_{a,b}&\leq  \sum_{a=1}^k\sum_{b\neq a} \opnorm{A-\E A}^2\norm{\theta_{a,b}}^2\\
&\leq \sum_{a=1}^k\sum_{b\neq a} 2c_1 np \norm{\theta_{a,b}}_1\\
&\leq 4c_1knp \norm{\pi-\Ztruth}_1.
\end{align*}
Thus, with probability at least $1-n^{-r}$,
\begin{align*}
L_2^\text{sum}&\leq \frac{4c_1knp \norm{\pi-\Ztruth}_1}{(\bar n_\textmin I/(4mt))^2}.
\end{align*}

\textit{(2) Sparse Scenario: $q^*<(\log n)/n$}. When the network is sparse, the previous upper bound on $\opnorm{A-\E A}$ no longer holds. Instead, removing nodes with large degrees is required to yield provably sharp bound on $\opnorm{A-\E A}$. Define $S=\{i\in[n],\sum_{j}A_{i,j}\geq 20np^*\}$.	We define $\tilde A,\tilde P$ such that $\tilde A_{i,j}= A_{i,j}\mathbb{I}\{i,j\notin S\}$ and $\tilde P_{i,j}=(\E A_{i,j })\mathbb{I}\{i,j\notin S\}$. Then we have the decomposition as
\begin{align*}
L_2(a,b)&\triangleq \sum_{i:\zi=b}\mathbb{I}\left[(A_{i,\cdot}-\E A_{i,\cdot})\theta_{a,b}\geq \frac{\bar n_\textmin I}{4mt}\right]\\
&\leq  \sum_{i:\zi=b}\mathbb{I}\left[(\tilde A_{i,\cdot}-\tilde P_{i,\cdot})\theta_{a,b}\geq \frac{\bar n_\textmin I}{8mt}\right]\\
&\quad + \sum_{i:\zi=b}\mathbb{I}\left[\sum_{j\neq i} (A_{i,j}-\E A_{i,j})[\theta_{a,b}]_{i,j}\mathbb{I}\{i\in S\text{ or }j\in S\}\geq \frac{\bar n_\textmin I}{8mt}\right]\\
&\triangleq L_{2,1}(a,b)+L_{2,2}(a,b).
\end{align*}
Define $L_{2,1}^\text{sum}\triangleq \sum_{a=1}^k\sum_{b\neq a} L_{2,1}(a,b)$. We have
\begin{align*}
L_{2,1}^\text{sum} \leq \sum_{a=1}^k\sum_{b\neq a} \frac{\theta_{a,b}^T [(\tilde A-\tilde P)^T(\tilde A-\tilde P)]\theta_{a,b}}{(\bar n_\textmin I/(8mt))^2}\leq \sum_{a=1}^k\sum_{b\neq a} \frac{2\opnorm{\tilde A - \tilde P}^2\norm{\theta_{a,b}}_1}{(\bar n_\textmin I/(8mt))^2}.
\end{align*}
Lemma \ref{lem:A_trunc_opnorm} shows $\opnorm{\tilde A-\tilde P}\leq \sqrt{c_2np}$ holds with probability at least $1-n^{-1}$ for some constant $c_2>0$. Then we have
\begin{align*}
L_{2,1}^\text{sum}\leq \frac{4c_2 knp \norm{\pi-\Ztruth}_1}{(\bar n_\textmin I/(8mt))^2}.
\end{align*}
Lemma \ref{lem:sum_zi} shows $\sum_{i,j}| A_{i,j}-\E A_{i,j}|\mathbb{I}\{i\in S\}\leq 20n^2p^*\exp(-5np^*)$ holds with probability at least $1-\exp(-5np^*)$. Then by applying Markov inequality, we have
\begin{align*}
L_{2,2}^\text{sum}&\triangleq \sum_{a=1}^k\left[\sum_{b\neq a}L_{2,2}(a,b)\right]\\
&\leq \sum_{a=1}^k\sum_{i,j=1}^n\frac{| A_{i,j}-\E A_{i,j}||[\theta_{a,b}]_{i,j}|\mathbb{I}\{i\in S\text{ or }j\in S\}}{\bar n_\textmin I/(8mt)}\\
&\leq \sum_{a=1}^k\frac{4\sum_{i,j}| A_{i,j}-\E A_{i,j}|\mathbb{I}\{i\in S\}}{\bar n_\textmin I/(8mt)}\\
&\leq  \frac{80 n^2kp^*\exp(-5np^*)}{\bar n_\textmin I/(8mt)}.
\end{align*}

As a consequence, we have
\begin{align*}
L_2^\text{sum}\leq L_{2,1}^\text{sum}+L_{2,2}^\text{sum}\leq \frac{4c_2 knp^* \norm{\pi-\Ztruth}_1}{(\bar n_\textmin I/(8mt))^2}+ \frac{80 n^2kp^*\exp(-5np^*)}{\bar n_\textmin I/(8mt)},
\end{align*}
with probability at least $1-n^{-1} -\exp(-5np^*)$. By the bounds on $L_1^\textsum$ and $L_2^\textsum$, and due to $t/t^*=1+o(1)$, we obtain Equation (\ref{eqn:upper_l2sum}).

\begin{supplement}
\sname{Supplement A}\label{suppA}
\stitle{Supplement to ``Theoretical and Computational Guarantees of Mean Field  Variational Inference for Community Detection''}
\slink[url]{url to be specified}
\sdescription{In the supplement \cite{zhang2017meanSupp}, we provide the detailed implementations of the batched Gibbs sampling and an iterative algorithm for MLE in Algorithm \ref{alg:gibbs} and Algorithm \ref{alg:mle} respectively. We include proof of Theorem \ref{thm:global}, Theorem \ref{thm:gibbs} and Theorem \ref{thm:mle}. We also include all the auxiliary propositions and lemmas in the supplement.
}
\end{supplement}

\bibliographystyle{plainnat}
\bibliography{variational}

\newpage
\thispagestyle{empty}
\setcounter{page}{1}
\begin{center}
\MakeUppercase{\large Supplement to ``Theoretical and Computational Guarantees of Mean Field Variational Inference for Community Detection''}
\medskip

{BY Anderson Y. Zhang and Harrison H.~Zhou}
\medskip

{Yale University}
\end{center}

\appendix

\section{Additional Algorithms}

In this section, we provide the detailed implementations of the batched Gibbs sampling and an iterative algorithm of MLE for community detection.

\subsection{Batched Gibbs Sampling}\label{sec:algo_gibbs}
~
\begin{algorithm}[ht]
\SetAlgoLined
\KwIn{Adjacency matrix $A$, number of communities $k$, hyperparameters $\pipr,\alphapr_p,\betapr_p,\alphapr_q,\betapr_q$, some initializers $Z^{(0)}$, number of iterations $S$.}
\KwOut{Gibbs sampling $\hat Z,\hat p,\hat q$.}
 \For {$s=1,2,\ldots, S$}  {
 \nl Update $\alpha_p^{(s)},\beta_p^{(s)},\alpha_q^{(s)},\beta_q^{(s)}$ by 
 \begin{align*}
  &\alpha^{(s)}_p = \alphapr_p + \sum_{a=1}^k\sum_{i<j}A_{i,j}Z^{(s-1)}_{i,a}Z^{(s-1)}_{j,a},\beta^{(s)}_p=\betapr_p+\sum_{a=1}^k\sum_{i<j}(1-A_{i,j})Z^{(s-1)}_{i,a}Z^{(s-1)}_{j,a},\\
  &\alpha^{(s)}_q = \alphapr_q + \sum_{a\neq b}\sum_{i<j}A_{i,j}Z^{(s-1)}_{i,a}Z^{(s-1)}_{j,b},\beta^{(s)}_q=\betapr_q+\sum_{a\neq b}\sum_{i<j}(1-A_{i,j})Z^{(s-1)}_{i,a}Z^{(s-1)}_{j,b}.
 \end{align*}
 
Then generate $p^{(s)}\sim\tbeta(\alpha^{(s)}_p,\beta^{(s)}_p)$ and $q^{(s)}\sim\tbeta(\alpha^{(s)}_q,\beta^{(s)}_q)$ independently.

\nl Define
\begin{align*}
t^{(s)}=\frac{1}{2}\log \frac{p^{(s)}(1-q^{(s)})}{(1-p^{(s)})q^{(s)}},\quad\text{and }\lambda^{(s)}=\frac{1}{2t^{(s)}}\log\frac{1-q^{(s)}}{1-p^{(s)}}.
\end{align*}

Then update $\pi^{(s)}$ with
\begin{align*}
\pi^{(s)}=h_{t^{(s)},\lambda^{(s)}}(Z^{(s-1)}),
\end{align*} 

where $h_{t,\lambda}(\cdot)$ is defined as in Equation (\ref{eqn:h}). Independently generate each row of $Z^{(s)}$ from distributions
\begin{align*}
\p(Z^{(s)}_{i,\cdot}=e_a)=\pi_{i,a}^{(s)},\forall a\in[k],\forall i\in[n].
\end{align*} 
 }
 \nl We have $\hat z = z^{(S)}$, $\hat p=p^{(S)}$ and $\hat q=q^{(S)}$.
\caption{Batched Gibbs Sampling \label{alg:gibbs}}
\end{algorithm}

\subsection{An Iterative Algorithm for Maximum Likelihood Estimation}\label{sec:alg_mle}

We first define a mapping $h':\Pi_0\rightarrow\Pi_0$ as follows
\begin{align}\label{eqn:h_prime}
[h'_{\lambda}(Z)]_{i,a} = \mathbb{I}\left[a=\argmax_b \sum_{j\neq i}Z_{i,b}(A_{i,j}-\lambda)\right].
\end{align}
Here if the maximizer is not unique, we simply pick the smallest index.

\begin{algorithm}[ht]
\SetAlgoLined
\KwIn{Adjacency matrix $A$, number of communities $k$, some initializers $z^{(0)}$, number of iterations $S$.}
\KwOut{Estimation $\hat Z,\hat p,\hat q$.}
 \For {$s=1,2,\ldots, S$}  {
 \nl Update $p^{(s)},q^{(s)}$ by
 \begin{align*}
 p^{(s)} = \frac{\sum_{a=1}^k\sum_{i<j}A_{i,j}Z^{(s-1)}_{i,a}Z^{(s-1)}_{j,a}}{\sum_{a=1}^k\sum_{i<j}(1-A_{i,j})Z^{(s-1)}_{i,a}Z^{(s-1)}_{j,a}}
 \end{align*}
 and
 \begin{align*}
 q^{(s)}=\frac{\sum_{a\neq b}\sum_{i<j}A_{i,j}Z^{(s-1)}_{i,a}Z^{(s-1)}_{j,b}}{\sum_{a\neq b}\sum_{i<j}(1-A_{i,j})Z^{(s-1)}_{i,a}Z^{(s-1)}_{j,b}}.
 \end{align*}
 
\nl Define
\begin{align*}
t^{(s)}=\frac{1}{2}\log \frac{p^{(s)}(1-q^{(s)})}{(1-p^{(s)})q^{(s)}},\quad\text{and }\lambda^{(s)}=\frac{1}{2t^{(s)}}\log\frac{1-q^{(s)}}{1-p^{(s)}}.
\end{align*}

Then update $\pi^{(s)}$ with
\begin{align*}
Z^{(s)}=h'_{\lambda^{(s)}}(Z^{(s-1)}),
\end{align*} 

where $h'_{\lambda}(\cdot)$ is defined as in Equation (\ref{eqn:h_prime}).
 }
 \nl We have $\hat z = z^{(S)}$, $\hat p=p^{(S)}$ and $\hat q=q^{(S)}$.
\caption{An Iterative Algorithm for MLE \label{alg:mle}}
\end{algorithm}

\section{Proofs of Other Theorems}
In this section, we first validate Theorem \ref{thm:mf_iterative} when $\ell(\pi^{(0)},\pi^*)$ is in a constant order of $\bar n_\textmin$, which complements the proof presented in Section \ref{sec:proof_mf_iterative}. The we give proofs of theorems stated in Section \ref{sec:diss}, including Theorem \ref{thm:global}, Theorem \ref{thm:gibbs} and Theorem \ref{thm:mle}.

\subsection{Proof of Theorem \ref{thm:mf_iterative} for the case $\ell(\pi^{(0)},\pi^*)$ in a constant order of $\bar n_\textmin$}\label{sec:proof_mf_iterative_constant}

For any $\pi$ such that $\ell(\pi,\pi^*)\leq c_\tinit \bar n_\textmin$, we are going to show when $c_\tinit$ is sufficiently small
\begin{align}\label{eqn:iterative_constant_main}
\ell(h_{t,\lambda}(\pi),Z^*)\leq n\exp(-\bar n_\textmin I/25) + \frac{\ell(\pi,Z^*)}{2\sqrt{nI/[wk[n/\bar n_\textmin]^2]}},
\end{align}
with probability at least $1-\exp(-\bar n_\textmin I/10)-n^{-r}$ for some constant $r>0$. If it holds, for any $\pi^{(0)}$ such that $\ell(\pi^{(0)},Z^*)=c\bar n_\textmin$ for some constant $c\leq c_\tinit$, the term $n\exp(-\bar n_\textmin I/25)$ is dominated by $\ell(\pi^{(0)},Z^*)/\sqrt{nI/[wk[n/\bar n_\textmin]^2]}$ which implies
\begin{align*}
\ell(\pi^{(1)},Z^*)\leq n\exp(-(1-\eta)/\bar n_\textmin I) + \frac{\ell(\pi^{(0)},Z^*)}{\sqrt{nI/[wk[n/\bar n_\textmin]^2]}}.
\end{align*}
It also implies $\ell(\pi^{(1)},Z^*)=o(\bar n_\textmin)$, which means after the first iteration, the results in Section \ref{sec:proof_mf_iterative} can be directly applied and the proof is complete.

~\\
\noindent The proof of Equation (\ref{eqn:iterative_constant_main}) mainly follows the proof of \emph{Part One} in Section \ref{sec:proof_mf_iterative}. We have
\begin{align*}
\norm{[h_{t,\lambda}( \pi)]_{i,\cdot}-\Ztruth_{i,\cdot}}_1&\leq 2w\sum_{a\neq \zi}1\wedge\exp\left[2t\sum_{j\neq i}(\pi_{j,a}-\pi_{j,\zi})(A_{i,j}-\lambda)\right].
\end{align*}
Note that the inequality $1\wedge \exp(-x)\leq f(x_0)+\mathbb{I}\{x\geq x_0\}$ holds for any $x_0\geq 0$. By taking $x_0= (n_a+n_{\zi})I/4$, we have
\begin{align*}
\norm{[h_{t,\lambda}( \pi)]_{i,\cdot}-\Ztruth_{i,\cdot}}_1 \leq 2w\sum_{a\neq z_i}\left[\exp\left[-\frac{(n_a+n_{\zi})I}{4}\right] + \mathbb{I}\left[\sum_{j\neq i}(\pi_{j,a}-\pi_{j,\zi})(A_{i,j}-\lambda) \geq -\frac{(n_a+n_{\zi})I}{8t}\right]\right],
\end{align*}
and consequently, 
\begin{align*}
\norm{h_{t,\lambda}( \pi)-\Ztruth}_1 &\leq 2wnk\exp(-\bar n_\textmin I/2) \\
&\quad +2w\sum_{a=1}^k\sum_{b\neq a}\sum_{i:z_i=b}\mathbb{I}\bigg[\sum_{j\neq i}(\pi_{j,a}-\pi_{j,b})(A_{i,j}-\lambda)\geq -\frac{(n_a+n_{\zi})I}{8t}\bigg]\Bigg].
\end{align*}
Define $\theta_{a,b}$ the same way as in Section \ref{sec:proof_mf_iterative}, and by the same argument, we have 
\begin{align*}
&\norm{h_{t,\lambda}( \pi)-\Ztruth}_1 \leq 2wnk\exp(-\bar n_\textmin I/2) +2w\sum_{a=1}^k\sum_{b\neq a}\sum_{i:z_i=b}\mathbb{I}\left[(A_{i,\cdot}-\E A_{i,\cdot})\theta_{a,b}\geq \frac{\bar n_\textmin I}{8t}\right]\\
&\quad +2w\sum_{a=1}^k\sum_{b\neq a}\sum_{i:z_i=b}\mathbb{I}\bigg[\sum_{j\neq i}(Z^*_{j,a}-Z^*_{j,b})(A_{i,j}-\lambda)\geq -\frac{(n_a+n_b)I}{4t}-\sum_{j\neq i}(\E A_{i,j}-\lambda)[\theta_{a,b}]_j\bigg].
\end{align*}

From Lemma \ref{lem:concentration_t_lambda}, when $c_\tinit$ is sufficiently small, with probability at least $1-e^35^{-n}$ we have
\begin{align}\label{eqn:proof_constant_t_lambda}
\max\left\{\frac{|t-t^*|}{(p^*-q^*)/p^*},\frac{|\lambda-\lambda^*|}{(p^*-q^*)}\right\} \leq 24c_0c_\tinit.
\end{align}
Proposition \ref{prop:lambda} shows that $\lambda^*\in(q^* + c(p^*-q^*), q^*+(1-c)(p^*-q^*))$ for some positive constant $0<c<1/2$. Therefore, when $c_\tinit$ is sufficiently small, we have $\lambda \in(q^*,p^*)$. Thus,
\begin{align*}
\left|\sum_{j\neq i}(\E A_{i,j}-\lambda)[\theta_{a,b}]_j\right|\leq (p^*-q^*) \norm{\theta_{a,b}}_1 \leq (p^*-q^*) \norm{\pi-Z^*}_1\leq c_\tinit (p^*-q^*) \bar n_\textmin,
\end{align*}
where we use Equation (\ref{eqn:theta_ab}). By Equations (\ref{eqn:I_t}) - (\ref{eqn:t_star_simple}), it is smaller than $(n_a+n_{z_i})/(8t)$ when $c_\tinit$ is sufficiently small. As a consequence, we have
\begin{align*}
&\norm{h_{t,\lambda}( \pi)-\Ztruth}_1 \leq 2wnk\exp(-\bar n_\textmin I/2) +2w\sum_{a=1}^k\sum_{b\neq a}\sum_{i:z_i=b}\mathbb{I}\left[(A_{i,\cdot}-\E A_{i,\cdot})\theta_{a,b}\geq \frac{\bar n_\textmin I}{8t}\right]\\
&\quad +2w\sum_{a=1}^k\sum_{b\neq a}\sum_{i:z_i=b}\mathbb{I}\bigg[\sum_{j\neq i}(Z^*_{j,a}-Z^*_{j,b})(A_{i,j}-\lambda)\geq -\frac{(n_a+n_b)I}{8t}\bigg].
\end{align*}
Define $L_1^\textsum=\sum_{a=1}^k\sum_{b\neq a}\sum_{i:z_i=b}\mathbb{I}\left[\sum_{j\neq i}(Z^*_{j,a}-Z^*_{j,b})(A_{i,j}-\lambda)\geq -(n_a+n_b)I/(8t)\right]$ and $L_2^\textsum=\sum_{a=1}^k\sum_{b\neq a}\sum_{i:z_i=b}\mathbb{I}\left[(A_{i,\cdot}-\E A_{i,\cdot})\theta_{a,b}\geq \bar n_\textmin I/(8t)\right]$. Our analysis on them is quite similar to that in Section \ref{sec:proof_mf_iterative}. By Markov inequality,
\begin{align*}
\E L_1^\textsum&=\sum_{a=1}^k\sum_{b\neq a}\sum_{i:z_i=b}\p\left[t^*\sum_{j\neq i}(Z^*_{j,a}-Z^*_{j,b})(A_{i,j}-\lambda)\geq -t^*(n_a+n_b)I/(8t)\right]\\
&\leq \sum_{a=1}^k\sum_{b\neq a}\sum_{i:z_i=b} \exp\left[\frac{t^*(n_a+n_b)I}{8t} -t^*(\lambda-\lambda^*)(n_a-n_b)\right]\E \exp\left[t^*\sum_{j\neq i}(Z^*_{j,a}-Z^*_{j,b})(A_{i,j}-\lambda^*)\right]\\
&\leq \sum_{a=1}^k\sum_{b\neq a}\sum_{i:z_i=b} \exp\left[\frac{t^*(n_a+n_b)I}{8t} -t^*(\lambda-\lambda^*)(n_a-n_b)-\frac{(n_a+n_b)I}{2}\right].
\end{align*}
By Equations (\ref{eqn:I_t}) -  (\ref{eqn:t_star_simple}) and (\ref{eqn:proof_constant_t_lambda}), when $c_\tinit$ is small enough, $t^*/t\leq 2$ and $t^*|\lambda-\lambda^*|\leq I/6$. Thus
\begin{align*}
\E L_1^\textsum\leq nk\exp(-\bar n_\textmin I/12).
\end{align*}
Hence, with probability at least $1-\exp(-\bar n_\textmin I/24)$,
\begin{align*}
L_1^\textsum\leq nk\exp(-\bar n_\textmin I/24).
\end{align*}
For $L_2^\textsum$ we use the same argument as in Section \ref{sec:proof_mf_iterative} and obtain
\begin{align*}
L_2^\textsum\leq \frac{4c_2 knp^* \norm{\pi-\Ztruth}_1}{(\bar n_\textmin I/(8t))^2}+ \frac{80 n^2kp^*\exp(-5np^*)}{\bar n_\textmin I/(8t)},
\end{align*}
with probability at least $1-n^{-r} -\exp(-5np^*)$ for some constants $r,c_1,c_2>0$. Recall that
\begin{align*}
\norm{h_{t,\lambda}( \pi)-\Ztruth}_1 \leq 2wnk\exp(-\bar n_\textmin I/2) + 2w L_1^\textsum + 2w L_2^\textsum.
\end{align*}
Using the same argument as in Section \ref{sec:proof_mf_iterative}, we conclude with
\begin{align*}
\norm{h_{t,\lambda}(\pi)-\Ztruth}_1\leq n\exp(-\bar n_\textmin I/25) + \frac{1}{2\sqrt{nI/[wk[n/\bar n_\textmin]^2]}}\norm{\pi-\Ztruth}_1,
\end{align*}
with probability at least $1-\exp(-\bar n_\textmin I/10)-n^{-r}$.

\subsection{Proof of Theorem \ref{thm:global}}
Define $t^*=\frac{1}{2}\log\frac{p^*(1-q^*)}{q^*(1-p^*)}$ and $\lambda^* =\frac{1}{2t^*}\log\frac{1-q^*}{1-p^*}$. By the same simplification we derive in Theorem \ref{thm:MF_simplifed}, we have
\begin{align*}
\hat\pi^\mf = \argmax_{\pi\in\Pi_1} f'(\pi;A),
\end{align*}
where
\begin{align*}
f'(\pi;A)=\langle A+\lambda^* I_n-\lambda^* 1_n1_n^T,\pi\pi^T\rangle -\frac{1}{t^*}\sum_{i=1}^n \kl(\tcate(\pi_{i,\cdot})\|\tcate(\pipr_{i,\cdot})).
\end{align*}
Recall the definition of $h_{t,\lambda}(\cdot)$ as in Equation (\ref{eqn:h}). A key observation is that $\hat\pi^\mf =h_{t^*,\lambda^*}(\hat\pi^\mf)$, otherwise if there exists some $i\in[n]$ such that $[h_{t^*,\lambda^*}(\hat\pi^\mf)]_{i,\cdot}$ not equal to $\hat\pi^\mf_{i,\cdot}$. This indicates the implementation of CAVI update on the $i$-th row of $\pi$ will make change, leading to the decrease of $f'(\cdot;A)$. This contradicts with the fact that $\hat\pi^\mf$ is the global minimizer.

The fixed-point property of $\hat\pi^\mf$ is the key to our analysis. It involves three steps.
\begin{itemize}
\item \emph{Step One.} For any $\pi$ such that $\ell(\pi,\Ztruth)=o(\bar n_\textmin)$, by the same analysis as in the proof of Theorem \ref{thm:mf_iterative}, we are able to show that there exist constant $r>0$ and sequence $\eta=o(1)$ such that
\begin{align*}
\norm{h_{t^*,\lambda^*}(\pi)-\Ztruth}_1 \leq n\exp(-(1-\eta)\bar n_\textmin I) + \frac{\norm{\pi-\Ztruth}_1}{\sqrt{nI/[wk[n/\bar n_\textmin]^2]}},
\end{align*}
with probability at least $1-\exp[-(\bar n_\textmin I)^\frac{1}{2}] -n ^{-r}$.

\item \emph{Step Two.} Lemma \ref{lem:mf_loose} presents some loose upper bound for $\ell(\hat\pi^\mf,\Ztruth)$. That is, under the assumption $\rho nI/[wk^2[n/\bar n_\textmin]^2]\rightarrow\infty$, with probability at least $1-e^35^{-n}$, we have
\begin{align*}
\ell(\hat\pi^\mf,\Ztruth) \leq o(\bar n_\textmin).
\end{align*}

\item \emph{Step Three.} Using the property that $h_{t^*,\lambda^*}(\hat\pi^\mf)=\hat\pi^\mf$, we have
\begin{align*}
\norm{\hat\pi^\mf-\Ztruth}_1 \leq n\exp(-(1-\eta)\bar n_\textmin I) + \frac{\norm{\hat\pi^\mf-\Ztruth}_1}{\sqrt{nI/[wk[n/\bar n_\textmin]^2]}}
\end{align*}
holds with probability at least $1-\exp[-(\bar n_\textmin I)^\frac{1}{2}] -n ^{-r}$. Then we obtain the desired result by simple algebra.

\end{itemize}

\subsection{Proof of Theorem \ref{thm:gibbs}}
By law of total expectation, we have
\begin{align}\label{eqn:gibbs_proof_1}
\E_{Z^{(s+1)}}\Big[\norm{Z^{(s+1)}-\Ztruth}_1\Big|A,Z^{(0)}\Big]&= \E_{\pi^{(s+1)}}\bigg[\E_{Z^{(s+1)}}\Big[\norm{Z^{(s+1)}-\Ztruth}_1\Big|\pi^{(s+1)},A,Z^{(0)}\Big]\bigg|A,Z^{(0)}\bigg]\\
& = \E_{\pi^{(s+1)}}\Big[\norm{\pi^{(s+1)}-\Ztruth}_1\Big|A,Z^{(0)}\Big],\nonumber
\end{align}
where the first equation is due to that the conditional expectation of $Z^{(s+1)}$ is $\pi^{(s+1)}$. We are going to build the connection between $\pi^{(s)}$ and $\pi^{(s+1)}$. In Algorithm \ref{alg:gibbs}, there are intermediate steps between $\pi^{(s)}$ and $\pi^{(s+1)}$ as follows:
\begin{align*}
\pi^{(s)} \leadsto Z^{(s)} \leadsto (p^{(s+1)},q^{(s+1)})\rightarrow (t^{(s+1)},\lambda^{(s+1)})\rightarrow\pi^{(s+1)},
\end{align*}
where we use the plain right arrow ($\rightarrow$) to indicate deterministic generation and the curved right arrow ($\leadsto$) to indicate random generation. Despite a slight abuse of notation, we define $\pi^{(0)}=Z^{(0)}$.

Analogous to the proof of Theorem \ref{thm:mf_iterative} in Section \ref{sec:proof_mf_iterative}, we assume $\ell(Z^{(0)},Z^*)=o(\bar n_\textmin)$. The proof for the case $\ell(Z^{(0)},Z^*)$ in the same order of $\bar n_\textmin$ is similar and thus is omitted.

Let $\gamma=o(1)$ be any sequence goes to 0 when $n$ grows. We define a series of events as follows:
\begin{itemize}
\item global event $\mathcal{F}$: We define $\mathcal{F}$ exactly the same way as we define in the proof of Theorem \ref{thm:mf_iterative} in Section \ref{sec:proof_mf_iterative} with respect to sequences $\gamma$ and $\eta'$, and we have $\p(\mathcal{F})\geq 1-\exp[-(\bar n_\textmin I)^\frac{1}{2})] -n ^{-r}$ for some constant $r>0$. We have $\eta'=o(1)$ whose value will be determined later.

\item global event $\mathcal{G}$: Consider any $Z\in\Pi_1$ such that $\norm{Z-\Ztruth}_1\leq \gamma\bar n_\textmin$. Define 
\begin{align*}
 &\alpha_p = \alphapr_p + \sum_{a=1}^k\sum_{i<j}A_{i,j}Z_{i,a}Z_{j,a},\beta_p=\betapr_p+\sum_{a=1}^k\sum_{i<j}(1-A_{i,j})Z_{i,a}Z_{j,a},\\
  &\alpha_q = \alphapr_q + \sum_{a\neq b}\sum_{i<j}A_{i,j}Z_{i,a}Z_{j,b},\beta_q=\betapr_q+\sum_{a\neq b}\sum_{i<j}(1-A_{i,j})Z_{i,a}Z_{j,b}.
\end{align*}
Define $\mathcal{G}$ be the event that
\begin{align*}
\max\left\{\left|\frac{\alpha_p}{\alpha_p+\beta_p}-p^*\right|, \left|\frac{\alpha_q}{\alpha_q+\beta_q}-q^*\right|\right\}\leq \eta'' (p^*-q^*)
\end{align*}
holds uniformly over all the eligible $Z$ for some sequence $\eta''=o(1)$. Then by the same analysis as in Lemma \ref{lem:concentration_t_lambda}, we have $\p(\mathcal{G})\geq 1-e^35^{-n}$.

\item local events $\{\mathcal{H}_1^{(s)}\}_{s=1}^S$: We define $\mathcal{H}_1^{(s)} = \{\norm{\pi^{(s)}-\Ztruth}_1\geq \gamma \bar n_\textmin/2\}$.

\item local events $\{\mathcal{H}_2^{(s)}\}_{s=1}^S$: We define $\mathcal{H}_2^{(s)} = \{\norm{Z^{(s)} - \Ztruth}_1 \geq \gamma\bar n_\textmin\}$.  For the conditional probability, we have
\begin{align*}
&\p(\mathcal{H}_2^{(s)}=1| \mathcal{H}_1^{(s)}=0)\\
&\leq \p\left[\left|\sum_{i=1}^n\left[\norm{Z_{i,\cdot}^{(s)} - \Ztruth_{i,\cdot}}_1-\norm{\pi_{i,\cdot}^{(s)} - \Ztruth_{i,\cdot}}_1\right]\right|\geq \gamma \bar n_\textmin - \norm{\pi^{(s)} - \Ztruth}_1\Bigg|\mathcal{H}_1^{(s)}=0\right]\\
&\leq \p\left[\left|\sum_{i=1}^n\left[\norm{Z_{i,\cdot}^{(s)} - \Ztruth_{i,\cdot}}_1-\norm{\pi_{i,\cdot}^{(s)} - \Ztruth_{i,\cdot}}_1\right]\right|\geq \gamma \bar n_\textmin/2\Bigg|\mathcal{H}_1^{(s)}=0\right]
\end{align*}
Since $\norm{\pi^{(s)}-\Ztruth}_1\leq \gamma \bar n_\textmin/2$  given $\mathcal{H}_1^{(s)}=0$ by Bernstein inequality,  we have
\begin{align*}
\p(\mathcal{H}_2^{(s)}=1| \mathcal{H}_1^{(s)}=0)&\leq \exp\left[-\frac{(\gamma \bar n_\textmin)^2/8}{\norm{\pi^{(s)} - \Ztruth}_1+\gamma \bar n_\textmin/6}\right]\\
&\leq \exp\left[-3(\gamma \bar n_\textmin)^2/16\right].
\end{align*}

\item local events $\{\mathcal{H}_3^{(s)}\}_{s=1}^S$: We define $\mathcal{H}_3^{(s)} = \{|t^{(s)}-t^*|\geq \eta'(p^*-q^*)/p^*,\text{ or }|\lambda^{(s)}-\lambda^*|\geq \eta' (p^*-q^*)\}$. If the global event $\mathcal{G}$ holds and the local event $\mathcal{H}_2^{(s)}$ does not hold, we have
\begin{align*}
\max\left\{\left|\frac{\alpha^{(s+1)}_p}{\alpha_p^{(s+1)}+\beta_p^{(s+1)}}-p^*\right|, \left|\frac{\alpha_q^{(s+1)}}{\alpha_q^{(s+1)}+\beta_q^{(s+1)}}-q^*\right|\right\}\leq \eta'' (p^*-q^*).
\end{align*}
Note that $\alpha_p^{(s+1)}+\beta_p^{(s+1)} = \alphapr_p+\betapr_p+\sum_{a=1}^k\sum_{i<j} Z_{i,a}^{(s)}Z_{j,a}^{(s)} \geq n^2/k$. Using the tail bound of Beta distribution (Lemma \ref{lem:concentration_beta}) we are able to show
\begin{align*}
&\p\left[\left|p^{(s+1)} - \frac{\alpha^{(s+1)}_p}{\alpha_p^{(s+1)}+\beta_p^{(s+1)}}\right|\geq \eta''(p^*-q^*)\Bigg|\mathcal{H}_2^{(s)}=0,\mathcal{G}=1\right]\\
&\leq \exp\left[-\eta''^{2}n^2\frac{(p^*-q^*)^2}{2p^*}\right]\\
&\leq \exp\left[-\eta''^{2}n^2I/2\right],
\end{align*}
where the last inequality is due to Proposition \ref{prop:I}. This leads to
\begin{align*}
\p\left[\left|p^{(s+1)} - p^*\right|\geq 2\eta''(p^*-q^*)\Big|\mathcal{H}_2^{(s)}=0,\mathcal{G}=1\right]\leq \exp\left[-\eta''^{2}n^2I/2\right].
\end{align*}
And similar result holds for $q^{(s+1)}$. Then by the same analysis as in the proof of Lemma \ref{lem:concentration_t_lambda}, $\max\{|p^{(s+1)}-p^*|,|q^{(s+1)}-q^*|\}\leq 2\eta''(p^*-q^*)$ leads to
\begin{align*}
\max\left\{\frac{|t^{(s+1)}-t^*|}{(p^*-q^*)/p^*},\frac{|\lambda^{(s+1)-\lambda^*}|}{p^*-q^*}\right\}\leq 16c_0\eta''.
\end{align*}
By taking $\eta'=16c_0\eta''$, we obtain
\begin{align*}
\p(\mathcal{H}_3^{(s+1)}=1|\mathcal{H}_2^{(s)}=0,\mathcal{G}=1)\leq 2\exp\left[-\eta''^{2}n^2I/2\right].
\end{align*}
\end{itemize}

Note that events $\mathcal{F}$ and $\mathcal{G}$ are about the adjacency matrix $A$. The events $\mathcal{H}_1^{(s)}, \mathcal{H}_2^{(s)}$ and $\mathcal{H}_3^{(s+1)}$ are for $\pi^{(x)},Z^{(s)}$ and $(p^{(s+1)}, q^{(s+1)})$ respectively. With all the above events defined, we can continue our analysis for Equation (\ref{eqn:gibbs_proof_1}). Under the event $\mathcal{F}\cap \mathcal{G} \cap (\mathcal{H}_1^{(s)}\cup \mathcal{H}_2^{(s)}\cup \mathcal{H}_3^{(s+1)})^C$ we have
\begin{align}\label{eqn:gibbs_pi_z}
\norm{\pi^{(s+1)}-\Ztruth}_1 &\leq n\exp(-(1-\eta)\bar n_\textmin I) + c_n\norm{\pi^{(s)}-\Ztruth}_1,
\end{align}
where $c_n= [nI/[wk[n/\bar n_\textmin]^2]]^{-1/2}$. As a consequence, under the event $\mathcal{F}\cap \mathcal{G} \cap (\prod_{v=0}^s\mathcal{H}_1^{(v)}\cup \mathcal{H}_2^{(v)}\cup \mathcal{H}_3^{(v+1)})^C$, we have
\begin{align*}
\norm{\pi^{(s+1)}-\Ztruth}_1 \leq n\exp(-(1-2\eta)\bar n_\textmin I) + c_n^s\norm{\pi^{(0)}-\Ztruth}_1.
\end{align*}
Therefore, we have
\begin{align}\label{eqn:gibbs_proof_2}
&\E_{\pi^{(s+1)}}\left[\norm{\pi^{(s+1)}-\Ztruth}_1\Big|\mathcal{H}_1^{(0)}=0,\mathcal{F}=1,\mathcal{G}=1\right] \leq n\exp(-(1-2\eta)\bar n_\textmin I) \\
&+ c_n^s\norm{\pi^{(0)}-\Ztruth}_1 + n\p\left[\prod_{v=1}^s\mathcal{H}_1^{(v)}\cup \mathcal{H}_2^{(v)}\cup \mathcal{H}_3^{(v+1)}\Big|\mathcal{H}_1^{(0)}=0,\mathcal{F}=1,\mathcal{G}=1\right].\nonumber
\end{align}
Due to the small value of $c_n$, if $\norm{\pi^{(s)}-\Ztruth}_1 \leq \gamma \bar n_\textmin$, Equation (\ref{eqn:gibbs_pi_z}) immediately implies $\norm{\pi^{(s+1)}-\Ztruth}_1 \leq \gamma\bar n_\textmin$. This implies that under the event $\mathcal{F}\cup \mathcal{G}$ we have
\begin{align*}
\mathcal{H}_1^{(s+1)} \subset \mathcal{H}_1^{(s)} \cup \mathcal{H}_2^{(s)} \cup \mathcal{H}_3^{(s+1)},\forall s\geq 0,
\end{align*} 
and consequently,
\begin{align*}
\prod_{v=0}^s\mathcal{H}_1^{(v)}\cup \mathcal{H}_2^{(v)}\cup \mathcal{H}_3^{(v+1)} \subset \mathcal{H}_1^{(0)} \prod_{v=0}^s\mathcal{H}_2^{(v)}\cup \mathcal{H}_3^{(v+1)},\forall s\geq 1.
\end{align*}
Thus,
\begin{align}\label{eqn:events_summation}
&\p\left[\prod_{v=0}^s\mathcal{H}_1^{(v)}\cup \mathcal{H}_2^{(v)}\cup \mathcal{H}_3^{(v+1)}\Big|\mathcal{H}_1^{(0)}=0,\mathcal{F}=1,\mathcal{G}=1\right] \\
&\leq \p\left[\prod_{v=0}^s\mathcal{H}_2^{(v)}\cup \mathcal{H}_3^{(v+1)}\Big|\mathcal{H}_1^{(0)}=0,\mathcal{F}=1,\mathcal{G}=1\right]\nonumber\\
&\leq \sum_{v=0}^s\p(\mathcal{H}_2^{(v)}=1| \mathcal{H}_1^{(v)}=0) + \sum_{v=0}^n \p(\mathcal{H}_3^{(v+1)}=1|\mathcal{H}_2^{(v)}=0,\mathcal{G}=1)\nonumber\\
&\leq  (s+1)\left[\exp\left[-3(\gamma \bar n_\textmin)^2/16\right]+ 2\exp\left[-\eta''^{2}n^2I/2\right]\right].\nonumber
\end{align}
Note that $\p(\mathcal{H}_1^{(0)}=0,\mathcal{F}=1,\mathcal{G}=1)\geq 1-\exp[-(\bar n_\textmin I)^\frac{1}{2})] -n ^{-r} -e^35^{-n} -\epsilon$. Recall we define $\pi^{(0)}=Z^{(0)}$. By Equations (\ref{eqn:gibbs_proof_1}), (\ref{eqn:gibbs_proof_2}) and (\ref{eqn:events_summation}), we have
\begin{align*}
\E_{Z^{(s+1)}}\Big[\norm{Z^{(s+1)}-\Ztruth}_1\Big|A,Z^{(0)}\Big] &\leq n\exp(-(1-2\eta)\bar n_\textmin I) + c_n^s\norm{Z^{(0)}-\Ztruth}_1 + (s+1)nb_n,
\end{align*}
with probability at least $1-\exp[-(\bar n_\textmin I)^\frac{1}{2})] -n ^{-r} -e^35^{-n} -\epsilon$, where $b_n=\exp\left[-3(\gamma \bar n_\textmin)^2/16\right]+ 2\exp\left[-\eta''^{2}n^2I/2\right]$.

\subsection{Proof of Theorem \ref{thm:mle}}\label{sec:proof_mle}
Note the similarity between Algorithm \ref{alg:mle} and Algorithm \ref{alg:BCAVI}. We can prove Theorem \ref{thm:mle} with almost the identical argument used in the proof of Theorem \ref{thm:mf_iterative}, thus omitted.

\section{Statements and Proofs of Auxiliary Lemmas and Propositions}

We include all the auxiliary propositions and lemmas in this section.

\subsection{Statements and Proofs of Lemmas and Propositions for Theorem \ref{thm:mf_iterative}}\label{sec:Auxiliary}

\begin{lemma}\label{lem:concentration_t_lambda}
Let $c_\tinit$ be some sufficiently small constant. Consider any $\pi\in\Pi_1$ such that $\norm{\pi-\Ztruth}_1\leq c_\tinit n/k$. Let $\alpha_p,\beta_p,\alpha_q,\beta_q,t,\lambda$ be the outputs after one step CAVI iteration from $\pi$ described in Algorithm \ref{alg:BCAVI}. That is, they are defined as Equations (\ref{eqn:alpha_p_prime}) - (\ref{eqn:lambda_prime}). Define
\begin{align*}
\hat p =\frac{\sum_{i<j}\sum_{a=1}^k \pi_{i,a}\pi_{j,a}A_{i,j}}{\sum_{i<j}\sum_{a=1}^k \pi_{i,a}\pi_{j,a}},\text{ and }\hat q =\frac{\sum_{i<j}\sum_{a\neq b}\pi_{i,a}\pi_{j,b}A_{i,j}}{\sum_{i<j}\sum_{a\neq b}\pi_{i,a}\pi_{j,b}}.
\end{align*}
Under the same assumption as in Theorem \ref{thm:mf_iterative}, there exists some sequence $\epsilon=o(1)$ such that with probability at least $1-e^35^{-n}$, the following inequality holds
\begin{align*}
\max \left\{\frac{|\hat p-p^*|}{p^*-q^*},\frac{|\hat q-q^*|}{p^*-q^*},\frac{|t-t^*|}{(p^*-q^*)/p^*},\frac{|\lambda-\lambda^*|}{p^*-q^*}\right\}\leq \epsilon + 24c_0\frac{\norm{\pi-\Ztruth}_1}{n/k},
\end{align*}
 uniformly over all the eligible $\pi$. In addition if we further assume $c_\tinit $ goes to 0, the LHS of the above inequality will be simply upper bounded by $\epsilon$.
\end{lemma}

\begin{proof}
We are going to obtain tight bounds on $|\hat p-p^*|$ and $|\hat q -q^*|$ first. Note that we have the ``variance-bias'' decomposition as in
\begin{align*}
|\hat p-p^*|\leq \frac{|\sum_{i<j}\sum_{a=1}^k \pi_{i,a}\pi_{j,a}(A_{i,j}-\E A_{i,j})|}{\sum_{i<j}\sum_{a=1}^k \pi_{i,a}\pi_{j,a}}+\left|\frac{\sum_{i<j}\sum_{a=1}^k \pi_{i,a}\pi_{j,a}\E A_{i,j}}{\sum_{i<j}\sum_{a=1}^k \pi_{i,a}\pi_{j,a}}-p^*\right|.
\end{align*}
We have concentration inequality holds for the numerator in the first term by Lemma \ref{lem:A_minus_EA_Grothendiecks}. That is, with probability at least $1-e^35^{-n}$, we have
\begin{align*}
\left|\sum_{i<j}\sum_{a=1}^k \pi_{i,a}\pi_{j,a}(A_{i,j}-\E A_{i,j})\right|=\left|\frac{1}{2}\langle A-\E A,\pi\pi^T\rangle\right|\leq 3n\sqrt{np^*}
\end{align*}
holds uniformly over all $\pi\in\Pi_1$. For the denominator, we have
\begin{align*}
\frac{n^2}{2}\geq \sum_{i<j}\sum_{a=1}^k \pi_{i,a}\pi_{j,a}=\frac{1}{2}\sum_{a=1}^k\norm{\pi_{\cdot,a}}_1^2\geq  \frac{n^2}{2k},
\end{align*}
since $\sum_{a=1}^k\norm{\pi_{\cdot,a}}_1=n$. Thus, we are able to obtain an upper bound on the first term as
\begin{align*}
\frac{|\sum_{i<j}\sum_{a=1}^k \pi_{i,a}\pi_{j,a}(A_{i,j}-\E A_{i,j})|}{\sum_{i<j}\sum_{a=1}^k \pi_{i,a}\pi_{j,a}}\leq 6\sqrt{\frac{k^2p^*}{n}}.
\end{align*}
For the second term, since $\E A_{i,j}=p^*\sum_{a=1}^k\Ztruth_{i,a}\Ztruth_{j,a} + q^*(1-\sum_{a=1}^k\Ztruth_{i,a}\Ztruth_{j,a})$, we have
\begin{align*}
\left|\frac{\sum_{i<j}\sum_{a=1}^k \pi_{i,a}\pi_{j,a}\E A_{i,j}}{\sum_{i<j}\sum_{a=1}^k \pi_{i,a}\pi_{j,a}}-p^*\right|&=(p^*-q^*)\frac{\left|\sum_{i<j}\left[\sum_{a=1}^k \pi_{i,a}\pi_{j,a}\right]\left[\sum_{a=1}^k 1-\Ztruth_{i,a}\Ztruth_{j,a}\right]\right|}{\sum_{i<j}\sum_{a=1}^k \pi_{i,a}\pi_{j,a}}\\
&=(p^*-q^*)\frac{\left|\langle \pi\pi^T,11^T-\Ztruth\Ztrutht \rangle\right|}{\sum_{i<j}\sum_{a=1}^k \pi_{i,a}\pi_{j,a}}\\
&=(p^*-q^*)\frac{\left|\langle \pi\pi^T-\Ztruth\Ztrutht ,11^T-\Ztruth\Ztrutht \rangle\right|}{\sum_{i<j}\sum_{a=1}^k \pi_{i,a}\pi_{j,a}},
\end{align*}
where in the last inequality we use the orthogonality between $\Ztruth\Ztrutht $ and $11^T-\Ztruth\Ztrutht $.
For its numerator, we have
\begin{align*}
\left|\langle \pi\pi^T-\Ztruth\Ztrutht ,11^T-\Ztruth\Ztrutht \rangle\right|&\leq \norm{\pi\pi^T-\Ztruth\Ztrutht }_1\\
&\leq \norm{\pi-\Ztruth}_1(\norm{\pi}_1+\norm{\Ztruth}_1)\\
&\leq \norm{\pi-\Ztruth}_1(2\norm{\Ztruth}_1+\norm{\pi-\Ztruth}_1)\\
&\leq 3n\norm{\pi-\Ztruth}_1.
\end{align*}
This leads to
\begin{align*}
\left|\frac{\sum_{i<j}\sum_{a=1}^k \pi_{i,a}\pi_{j,a}\E A_{i,j}}{\sum_{i<j}\sum_{a=1}^k \pi_{i,a}\pi_{j,a}}-p^*\right|\leq \frac{3n\norm{\pi-\Ztruth}_1(p^*-q^*)}{n^2/k}\leq 3kn^{-1}(p^*-q^*)\norm{\pi-\Ztruth}_1.
\end{align*}
Thus,
\begin{align*}
|\hat p-p^*|&\leq 6\sqrt{\frac{k^2p^*}{n}} +3kn^{-1}(p^*-q^*)\norm{\pi-\Ztruth}_1 \leq\left[\sqrt{\frac{k^2p^*}{n(p^*-q^*)^2}}+\frac{3\norm{\pi-\Ztruth}_1}{n/k}\right](p^*-q^*).
\end{align*}
Similar result holds for $|\hat q-q^*|$. Denote $\eta_0 = \sqrt{\frac{k^2p^*}{n(p^*-q^*)^2}}+\frac{3\norm{\pi-\Ztruth}_1}{n/k}$, thus
\begin{align*}
\max\{|\hat p-p^*|,|\hat q-q^*|\} \leq \eta_0(p^*-q^*).
\end{align*}
By the assumption of $nI$ in Equation (\ref{eqn:assumption_nI}) and Proposition \ref{prop:I}, we have $n(p^*-q^*)^2/(k^2p^*)\asymp nI/k^2\rightarrow\infty$. Therefore, the first term in $\eta_0$ goes to 0. The second term in $\eta_0$ is at most $3c_\tinit$ which implies $\eta_0\leq 4c_\tinit$.

 By the fact that the digamma function satisfies $\psi(x)\in(\log(x-1/2),\log x),\forall x\geq 1/2$, we have
\begin{align*}
\psi(\alpha_p)-\psi(\beta_p)&\geq \log \frac{\alpha_p-1/2}{\beta_p}\\
&=\log\left[\frac{\left[\alphapr_p-1/2+\sum_{i<j}\sum_{a=1}^k \pi_{i,a}\pi_{j,a}A_{i,j}\right]\big/\left[\sum_{i<j}\sum_{a=1}^k \pi_{i,a}\pi_{j,a}\right]}{1+\left[\betapr_p-\sum_{i<j}\sum_{a=1}^k \pi_{i,a}\pi_{j,a}A_{i,j}\right]\big/\left[\sum_{i<j}\sum_{a=1}^k \pi_{i,a}\pi_{j,a}\right]}\right]\\
&=\log\left[\frac{\hat p+(\alphapr_p-1/2)\big/\left[\sum_{i<j}\sum_{a=1}^k \pi_{i,a}\pi_{j,a}\right]}{1-\hat p +\betapr_p\big/\left[\sum_{i<j}\sum_{a=1}^k \pi_{i,a}\pi_{j,a}\right]}\right].
\end{align*}
Recall that we have shown $\sum_{i<j}\sum_{a=1}^k \pi_{i,a}\pi_{j,a}$ lies in the interval of $(n^2/(2k),n^2/2)$. By Equation (\ref{eqn:assumption_nI}), there exists a sequence $\eta'=o(1)$ such that $\alpha_p,\beta_p\leq \eta'(p^*-q^*)n^2/k$. Then we have
\begin{align*}
\psi(\alpha_p)-\psi(\beta_p) \geq \log\frac{p^* - |p^*-\hat p|- \eta'(p^*-q^*)}{1-p^* + |p^*-\hat p|+ \eta'(p^*-q^*)}.
\end{align*}
Similar analysis leads to
\begin{align*}
\psi(\alpha_q)-\psi(\beta_q) \leq \log\frac{q^* + |q^*-\hat q| + \eta'(p^*-q^*)}{1-q^* - |q^*-\hat q| -\eta'(p^*-q^*)}.
\end{align*}
Together we have
\begin{align*}
t - t^* &\geq \log\left[\frac{p^* - |p^*-\hat p|- \eta'(p^*-q^*)}{1-p^* + |p^*-\hat p|+ \eta'(p^*-q^*)}\frac{1-q^* - |q^*-\hat q| -\eta'(p^*-q^*)}{q^* + |q^*-\hat q| + \eta'(p^*-q^*)}\right] - t^*\\
&\geq \log\left[\left[1-\frac{|p^*-\hat p|+ \eta'(p^*-q^*)}{q^*}\right]^4 \frac{p^*(1-q^*)}{q^*(1-p^*)}\right] -t^*\\
& = 4\log\left[1-(\eta_0+\eta')\frac{p^*-q^*}{q^*}\right].\\
\end{align*}
Recall that we assume $c_0p^*<q^*<p^*$. Thus $(\eta_0+\eta')(p^*-q^*)/p^*\leq 5c_\tinit c_0$. When $c_\tinit$ is sufficiently small, we have $(\eta_0+\eta')(p^*-q^*)/p^* \leq 1/2$. Then using the fact $-x\geq\log(1-x)\geq -2x,\forall x\in(0,1/2)$. 
We have
\begin{align*}
t - t^* \geq -8 (\eta_0+\eta')(p^*-q^*)/q^*.
\end{align*}
Analogously we can obtain the same upper bound on $\hat t - t^*$, and then 
\begin{align*}
|t - t^*| \leq 8c_0(\eta_0+\eta')\frac{p^*-q^*}{p^*}.
\end{align*}

Identical analysis can be applied towards bounds on $|\hat\lambda-\lambda^*|$. Note that
\begin{align*}
\log\frac{\beta_p}{\alpha_p+\beta_p}=\log\left[\frac{1-\hat p +\betapr_p\big/\left[\sum_{i<j}\sum_{a=1}^k \pi_{i,a}\pi_{j,a}\right]}{1 +(\alphapr_p+\betapr_p)\big/\left[\sum_{i<j}\sum_{a=1}^k \pi_{i,a}\pi_{j,a}\right]}\right],
\end{align*}
similarly for $\alpha_q,\beta_q$. Omitting the immediate steps, we end up with
\begin{align*}
|\lambda-\lambda^*| &= |\left[\psi(\beta_q)-\psi(\alpha_q+\beta_q)\right]- \left[\psi(\beta_p)-\psi(\alpha_p+\beta_p)\right] - \lambda^*| \leq 8(\eta_0+\eta')(p^*-q^*).
\end{align*}

The proof is complete after we unify and rephrase all the aforementioned results.
\end{proof}

\begin{lemma}\label{lem:A_minus_EA_Grothendiecks}
Let $A\in [0,1]^{n\times n}$ such that $A=A^T$ and $A_{i,i}=0,\forall i\in[n]$. Assume $\{A_{i,j}\}_{i<j}$ are independent random variable, and there exists $p\leq 1$ such that $9n^{-1}\leq \frac{2}{n(n-1)}\sum_{i<j} \text{Var}(A_{i,j})\leq p$, and then we have
\begin{align*}
\sup_{\pi\in\Pi_1}\Big|\langle A-\E A, \pi\pi^T\rangle\Big|\leq 6n\sqrt{np},
\end{align*}
with probability at least $1-e^35^{-n}$.
\end{lemma}
\begin{proof}
This result is a direct consequence of Grothendieck inequality \cite{grothendieck1996resume} (see also Theorem 3.1 of \cite{guedon2016community} for a rephrased statement) on the matrix $A-\E A$. The Lemma 4.1 of \cite{guedon2016community} proves that with probability at least $1-e^35^{-n}$,
\begin{align*}
\sup_{s,t\in\{-1,1\}^n}\Big|\sum_{i,j}(A_{i,j}-\E A_{i,j})s_it_j\Big|\leq 3n\sqrt{np}.
\end{align*}
Then by  applying Grothendieck inequality we obtain
\begin{align*}
\sup_{\norm{X_i}_2\leq 1,\forall i\in[n]}\Big|\sum_{i,j}(A_{i,j}-\E A_{i,j}) X_i^TX_j\Big|\leq 3cn\sqrt{np},
\end{align*}
where $c$ is a positive constant smaller than 2. This concludes with
\begin{align*}
\sup_{\pi\in\Pi_1}\Big|\langle A-\E A, \pi\pi^T\rangle\Big|\leq 6n\sqrt{np},
\end{align*}
\end{proof}

\begin{proposition}\label{prop:chernoff_I}
Assume $0<q<p<1$. Let $X\sim \Ber(q)$ and $Y\sim \Ber(p)$. Recall the definition $\lambda =\log\frac{1-q}{1-p}/\log\frac{p(1-q)}{q(1-p)}$, $t=\frac{1}{2}\log\frac{p(1-q)}{q(1-p)}$ and $I=-2\log[\sqrt{pq}+\sqrt{(1-p)(1-q)}]$. Then the following two equations hold
\begin{align}\label{eqn:chernoff_I_1}
e^{t\lambda} = \left(\frac{\E e^{tX}}{\E e^{-tY}}\right)^\frac{1}{2},\text{ and }\;\;\E e^{tX}\E e^{-tY} =\exp(-I).
\end{align}
\end{proposition}
\begin{proof}
The proof is straightforward and all by calculation. Note that $\E \exp(tX)=pe^t+1-p$ and $\E \exp(tY)=qe^t+1-q$. We can easily obtain
\begin{align*}
\E e^{tX}\E e^{-tY} = (pe^t+1-p)(qe^{-t}+1-q)=(\sqrt{pq}+\sqrt{(1-p)(1-q)})^2=\exp(-I).
\end{align*}
We can justify the first part of Equation \eqref{eqn:chernoff_I_1} in a similar way. 
\end{proof}

\begin{lemma}\label{lem:A_opnorm}[Theorem 5.2 of \cite{lei2015consistency}]
Let $A\in\{0,1\}^{n\times n}$ be a symmetric binary matrix with $A_{i,i}=0,\forall i\in[n]$, and $\{A_{i,j}\}_{i<j}$ are independent Bernoulli random variable. If $p\triangleq\max_{i,j} \E A_{i,j} \geq \log n/n$. Then there exist constants $c,r>0$ such that
\begin{align*}
\opnorm{A-\E A}\leq c\sqrt{np},
\end{align*}
with probability at least $1-n^{-r}$.
\end{lemma}

The following lemma on the operator norm of sparse networks is from \cite{chin2015stochastic}. In the original statement of Lemma 12 in \cite{chin2015stochastic}, ``with probability $1-o(1)$'' is stated. However, its proof in \cite{chin2015stochastic} gives explicit form of the probability that the statement holds, which is at least $1-n^{-1}$.
\begin{lemma}\label{lem:A_trunc_opnorm}[Lemma 12 of \cite{chin2015stochastic}]
Suppose $M$ is random symmetric matrix with zero on the diagonal whose entries above the diagonal are independent with the following distribution
\begin{align*}
M_{i,j}=\begin{cases}
1-p_{i,j},\text{ w.p. }p_{i,j};\\
-p_{i,j},\text{ w.p. }1-p_{i,j}.
\end{cases}
\end{align*}
Let $p\triangleq\max_{i,j} p_{i,j}$ and $\tilde M$ be the matrix obtained from $M$ by zeroing out all the rows and columns having more than $20np$ positive entries. Then there exists some constant $c>0$ such that 
\begin{align*}
\opnorm{\tilde M}\leq c\sqrt{np},
\end{align*}
holds with probability at least $1-n^{-1}$.
\end{lemma}

\begin{lemma}\label{lem:sum_zi}
Let $A\in\{0,1\}^{n\times n}$ be a symmetric binary matrix with $A_{i,i}=0,\forall i\in[n]$, and $\{A_{i,j}\}_{i<j}$ are independent Bernoulli random variable. Let $p\geq \max_{i,j} \E A_{i,j}$. Define $S=\{i\in[n],\sum_{j}A_{i,j}\geq 20np\}$ and $Z_i=\sum_j|A_{i,j}-\E A_{i,j}|\mathbb{I}\{i\in S\}$. Then with probability at least $1-\exp(-5np)$, we have
\begin{align*}
\sum_{i} Z_i\leq 20n^2p\exp(-5np).
\end{align*}
\end{lemma}
\begin{proof}
Note that $\E \sum_j|A_{i,j}-\E A_{i,j}| \leq 2np(1-p)\leq 2np$. For any $s\geq 20np$, we have
\begin{align*}
\p(Z_i> s)&\leq \p\left[\sum_j|A_{i,j}-\E A_{i,j}|-\E \sum_j|A_{i,j}-\E A_{i,j}|>s-2np\right]\\
&\leq \exp\left[-\frac{\frac{1}{2}(s-2np)^2}{np + \frac{1}{3}(s-2np)}\right]\\
&\leq \exp(-s/2),
\end{align*}
by implementing Bernstein inequality. Applying Bernstein inequality again we have
\begin{align*}
\p(Z_i>0) &= \p\left[\sum_{j} A_{i,j}\geq 20np\right]\\
&\leq \p\left[\sum_j A_{i,j}-\E \sum_j A_{i,j}\geq 18 np\right]\\
&\leq \exp\left[-\frac{(18np)^2/2}{np+18np/3}\right]\\
&\leq \exp(-21np/2).
\end{align*}
Thus, we are able to bound $\E Z_i$ with
\begin{align*}
\E Z_i&\leq \int_0^{20np}\p(Z_i>0)\diff s +\int_{20np}^\infty \p(Z_i>s)\diff s\\
&\leq 20np\exp(-21np/2) + \int_{20np}^\infty \exp(-s/2)\\
&\leq 20np\exp(-10np).
\end{align*}
By Markov inequality, we have
\begin{align*}
\p\left[\sum_{i,j}| A_{i,j}-\E A_{i,j}|\mathbb{I}\{i\in S\}\geq 20n^2p\exp(-5np)\right]&=\p\left[\sum_i Z_i\geq 20n^2p\exp(-5np)\right]\\
&\leq \frac{n\E Z_1}{20n^2p\exp(-5np)}\\
&\leq \exp(-5np).
\end{align*}
\end{proof}

\begin{proposition}\label{prop:I}
Under the assumption that $0<q<p=o(1)$. For $I=-2\log\left[\sqrt{pq}+\sqrt{(1-p)(1-q)}\right]$ we have
\begin{align*}
I=(1+o(1))(\sqrt{p}-\sqrt{q})^2.
\end{align*}
Consequently, $(p-q)^2/(4p)\leq I\leq (p-q)^2/p$.
\end{proposition}
\begin{proof}
It is a partial result of Lemma B.1 in \cite{zhang2016minimax}.
\end{proof}

\begin{proposition}\label{prop:lambda}
Define $\lambda =\log\frac{1-q}{1-p}/\log\frac{p(1-q)}{q(1-p)}$. For any $p,q>0$ such that $p,q=o(1)$ and $p\asymp q$, there exists a constant $0<c<1/2$ such that
\begin{align*}
\frac{\lambda -q}{p-q}\in(c,1-c).
\end{align*}
\end{proposition}
\begin{proof}
First we are going to establish the lower bound. Let $x=p-q$, and then we can rewrite $\lambda$ as
\begin{align*}
\lambda = \frac{1}{1+\frac{\log(1+x/q)}{\log(1+x/(1-q-x))}}.
\end{align*}
\paragraph{Case I: $x\geq q/10$} Define $s = (p-q)/q$. Since $p\asymp q$ we have $s\geq 1/10$ and also upper bounded by some constant. We have
\begin{align*}
\frac{\lambda -q}{p-q}&=\frac{1}{s}\left[\frac{1}{q}\frac{1}{1+\frac{\log(1+s)}{\log(1+sq/(1-(s+1)q))}}-1\right]\\
&= \frac{1}{s}\left[\frac{(1-q)\log(1+sq/(1-(s+1)q))-q\log(1+s)}{q\log(1+sq/(1-(s+1)q))+q\log(1+s)}\right]\\
&\geq \frac{1}{s}\frac{(1-q)\frac{sq}{1-(s+1)q}-q\log(1+s)}{2q\log(1+s)}\\
&\geq \frac{1}{8}\frac{1-q}{\log(1+s)},
\end{align*}
which is lower bounded by some constant $c>0$.
\paragraph{Case II: $x < q/10$} By Taylor theorem, there exist constants $0\leq \epsilon_1,\epsilon_2\leq 1/10$ such that 
\begin{align*}
&\log\left[1+\frac{x}{q}\right] = \frac{x}{q}-\frac{1-\epsilon_1}{2}\left[\frac{x}{q}\right]^2,\\
\text{and }&\log\left[1+\frac{x}{1-q-x}\right] = \frac{x}{1-q-x}-\frac{1-\epsilon_2}{2}\left[\frac{x}{1-q-x}\right]^2.
\end{align*}
Thus, we have
\begin{align*}
\frac{\log(1+\frac{x}{q})}{\log(1+\frac{x}{1-q-x})}=\frac{q(1-q)^2-\left[2q(1-q)+\frac{1-\epsilon_1}{2}(1-q)^2\right]x+c_1 x^2 +c_2 x^3}{q^2(1-q)-\frac{3-\epsilon_2}{2}q^2x},
\end{align*}
where $c_1=(1-\epsilon_1)(1-q)+q$ and $c_2=-(1-\epsilon_1)/2$. Thus,
\begin{align*}
\frac{\lambda -q}{p-q}& = \frac{1}{x}\left[\frac{q^2(1-q)-\frac{3-\epsilon_2}{2}q^2x}{q(1-q)-\left[2q(1-q)+\frac{1-\epsilon_1}{2}(1-q)^2+\frac{3-\epsilon_2}{2}q^2\right]x+c_1 x^2 +c_2 x^3}-q\right]\\
&=\frac{\left[\frac{1}{2}q(1-q)+\frac{\epsilon_2}{2}q^2(1-q)-\frac{\epsilon_1}{2}(1-q)^2q\right]+c_1qx+c_2qx^2}{q(1-q)-\left[2q(1-q)+\frac{1-\epsilon_1}{2}(1-q)^2+\frac{3-\epsilon_2}{2}q^2\right]x+c_1 x^2 +c_2 x^3}
\end{align*}
Note that $|c_1|,|c_2|\leq 1$. We have
\begin{align*}
\frac{\lambda -q}{p-q} \geq \frac{\frac{1}{4}q(1-q)}{2q(1-q)}\geq 1/8.
\end{align*}

By using exactly the same discussion, we can show $(p-\lambda)/(p-q)>c$. Thus, we proved the desired bound stated in the proposition.
\end{proof}

\subsection{Statements and Proofs of Lemmas and Propositions for Theorem \ref{thm:global}}

\begin{lemma}\label{lem:mf_loose}
Let $\Ztruth\in\Pi_0$. Assume $p^*,q^*=o(1)$ and $p^*\asymp q^*$. Define $t^*,\lambda^*$ and $\hat\pi^\mf$ the same way as in Theorem \ref{thm:global}. If $nI/[k\log kw]\rightarrow\infty$, we have with probability at least $1-e^35^{-n}$,
\begin{align*}
\norm{\Ztruth\Ztrutht -\hat\pi^\mf(\hat\pi^{\mf})^T}_1\lesssim n^2/\sqrt{nI}.
\end{align*}
If we further assume $\Ztruth\in\Pi_0^{(\rho,\rho')}$ with arbitrary $\rho,\rho'$, and then we have with probability at least $1-e^35^{-n}$,
\begin{align*}
\ell(\hat\pi^\mf,\Ztruth)\lesssim \rho^{-1}n\sqrt{k^2/(nI)}.
\end{align*}
\end{lemma}

\begin{proof}
Form Lemma \ref{lem:A_minus_EA_Grothendiecks}, with probability at least $1-e^35^{-n}$, we have uniformly for all $\pi\in\Pi_1$
\begin{align}\label{eqn:grothendieck}
|\langle A-\E A, \pi\pi^T\rangle|\leq 6n\sqrt{np^*}.
\end{align} 
In the remaining part of the proof, we always assume the above event holds. Denote $f'(\pi)=  \langle A+\lambda^* I_n-\lambda^* 1_n1_n^T,\pi\pi^T\rangle- (t^*)^{-1}\sum_{i=1}^n \kl(\pi_{i,\cdot}\|\pipr_{i,\cdot})$ for any $\pi\in\Pi_1$. Here we adopt the notation $\kl(\pi_{i,\cdot}\|\pipr_{i,\cdot})$ short for $\kl(\tcate(\pi_{i,\cdot})\|\tcate(\pipr_{i,\cdot}))$, and we do it in the same way in the rest part of the proof.  Thus,
\begin{align*}
\langle \E A+\lambda^* I_n -\lambda^* 1_n1_n^T,\hat\pi^\mf(\hat\pi^{\mf})^T\rangle &\geq \langle A+\lambda^* I_n -\lambda^* 1_n1_n^T,\hat\pi^\mf(\hat\pi^{\mf})^T\rangle -6n\sqrt{np^*}\\
&= f'(\hat\pi^\mf) -6n\sqrt{np^*} + (t^*)^{-1}\sum_{i=1}^n \kl(\hat\pi_{i,\cdot}^\mf\|\pipr_{i,\cdot})\\
&\geq f'(\Ztruth)-6n\sqrt{np^*}+ (t^*)^{-1}\sum_{i=1}^n \kl(\hat\pi_{i,\cdot}^\mf\|\pipr_{i,\cdot})\\
&\geq \langle \E A+\lambda^* I_n -\lambda^* 1_n1_n^T,\Ztruth\Ztrutht \rangle -12n\sqrt{np^*}\\
&\quad+(t^*)^{-1}\sum_{i=1}^n \kl(\hat\pi_{i,\cdot}^\mf\|\pipr_{i,\cdot})-(t^*)^{-1}\sum_{i=1}^n \kl(\Ztruth_{i,\cdot}\|\pipr_{i,\cdot}),
\end{align*}
where we use Equation \eqref{eqn:grothendieck} twice in the first and last inequality. Note that for any $\pi\in\Pi_1$, we have
\begin{align*}
|\kl(\pi_{i,\cdot}\|\pipr_{i,\cdot})|\leq |\sum_{j}\pi_{i,j}\log\pi_{i,j}|+|\sum_{j}\pi_{i,j}\log\pipr_{i,j}|\leq \log k + \log w,
\end{align*}
where the second inequality is due to $0\geq \sum_{j}\pi_{i,j}\log\pi_{i,j} = \kl(\pi_{i,\cdot}\|k^{-1} 1_k)-\log k\geq -\log k,$ where $k^{-1} 1_k$ can be explicitly written as a length-$k$ vector $(1/k,1/k,\ldots,1/k)$. Then we have
\begin{align*}
\Bigg|\sum_{i=1}^n \kl(\hat\pi^\mf_{i,\cdot}\|\pipr_{i,\cdot})- \sum_{i=1}^n \kl(\Ztruth_{i,\cdot}\|\pipr_{i,\cdot})\Bigg|\leq 2n\log kw.
\end{align*}
Thus,
\begin{align*}
\langle \E A+\lambda^* I_n -\lambda^* 1_n1_n^T,\Ztruth\Ztrutht -\hat\pi^\mf(\hat\pi^{\mf})^T\rangle \leq 12n\sqrt{np^*} +2(t^*)^{-1}n\log kw.
\end{align*}
By Proposition \ref{prop:alpha_plus_gamma}, we have
\begin{align*}
\langle \E A+\lambda^* I_n -\lambda^* 1_n1_n^T,\Ztruth\Ztrutht -\hat\pi^\mf(\hat\pi^{\mf})^T\rangle\geq 2(p^*-q^*)\left[\left(1-\frac{\lambda^*-q^*}{p^*-q^*}\right)\alpha+\frac{\lambda^*-q^*}{p^*-q^*}\gamma\right],
\end{align*}
where $\alpha=\langle\Ztruth\Ztrutht -\hat\pi^\mf(\hat\pi^\mf)^{T},\Ztruth\Ztrutht -I_n\rangle/2$ and $\gamma=\langle\hat\pi^\mf(\hat\pi^\mf)^{T}-\Ztruth\Ztrutht ,1_n1_n^T-\Ztruth\Ztrutht \rangle/2$. By Proposition \ref{prop:lambda}, there exists a constant $c>0$ such that
\begin{align}\label{eqn:inner_product_lower}
\langle \E A+\lambda^* I_n -\lambda^* 1_n1_n^T,\Ztruth\Ztrutht -\hat\pi^\mf(\hat\pi^{\mf})^T\rangle \geq 2c(p^*-q^*)(\alpha+\gamma).
\end{align}
Note that the following inequality holds
\begin{align*}
2(\alpha+\gamma) &=\norm{\Ztruth\Ztrutht -\hat\pi^\mf(\hat\pi^\mf)^{T}}_1-\langle \Ztruth\Ztrutht -\hat\pi^\mf(\hat\pi^\mf)^{T},I_n \rangle/2\\
&\geq \norm{\Ztruth\Ztrutht -\hat\pi^\mf(\hat\pi^\mf)^{T}}_1 - n/2.
\end{align*}
These together lead to
\begin{align*}
\norm{\Ztruth\Ztrutht -\hat\pi^\mf(\hat\pi^{\mf})^T}_1\leq \frac{1}{c(p^*-q^*)}\left[12n\sqrt{np^*} +2(t^*)^{-1}n\log kw+c(p^*-q^*)n/2\right].
\end{align*}
Note that $t^*\asymp (p^*-q^*)/p^*$ when $p^*\asymp q^*$. Together by Proposition \ref{prop:I}, as long as $nI/[k\log kw]\rightarrow\infty$, the last two terms in the RHS of the above formula is dominated by the first term. Thus,
\begin{align*}
\norm{\Ztruth\Ztrutht -\hat\pi^\mf(\hat\pi^{\mf})^T}_1\lesssim \frac{n^2}{\sqrt{nI}}.
\end{align*}
If we further assume $\Ztruth\in\Pi_0^{(\rho,\rho')}$, Proposition \ref{prop:alpha_gamma} and  Equation \eqref{eqn:inner_product_lower} lead to
\begin{align*}
\langle \E A+\lambda^* I_n -\lambda^* 1_n1_n^T,\Ztruth\Ztrutht -\hat\pi^\mf(\hat\pi^{\mf})^T\rangle \geq \frac{\rho c n(p^*-q^*)}{8k}\ell(\hat\pi^\mf,\Ztruth).
\end{align*}
So we have
\begin{align*}
\ell(\hat\pi^\mf,\Ztruth) &\leq \frac{8k}{\rho cn(p^*-q^*)}(12n\sqrt{np^*} +2(t^*)^{-1}n\log kw)\\
&\leq \frac{192 k}{\rho c}\sqrt{\frac{np^*}{(p^*-q^*)^2}}.
\end{align*}
\end{proof}

Before we state the remaining lemmas and propositions used in the Proof of Lemma \ref{lem:mf_loose}, we first introduce two definitions. For any $\pi,\pi'\in[0,1]^{n\times k}$, define $\alpha(\pi;\pi')=\langle\pi^{'}\pi^{'T}-\pi\pi^T,\pi^{'}\pi^{'T}-I_n\rangle/2$ and $\gamma(\pi;\pi')=\langle\pi\pi^T-\pi^{'}\pi^{'T},1_n1_n^T-\pi^{'}\pi^{'T}\rangle/2$.

\begin{proposition}\label{prop:alpha_plus_gamma}
Define $P = \Ztruth B\Ztrutht -p I_n$, with $B=q1_k1_k^T+(p-q)I_k$. We have the equation
\begin{align*}
\langle P+\lambda I_n -\lambda 1_n1_n^T, \Ztruth\Ztrutht  - \pi\pi^T \rangle = 2(p-q)\left[\left(1-\frac{\lambda-q}{p-q}\right)\alpha(\pi;\Ztruth)+\frac{\lambda-q}{p-q}\gamma(\pi;\Ztruth)\right].
\end{align*}
\end{proposition}
\begin{proof}
Note that $\Ztruth B\Ztrutht -p I_n = (p-q) \Ztruth\Ztrutht  + q 1_n1_n^T$. We have
\begin{align*}
\langle P+\lambda I_n -\lambda 1_n1_n^T, \Ztruth\Ztrutht  - \pi\pi^T \rangle&= (p-q)\langle \Ztruth\Ztrutht -\frac{\lambda -q }{p-q}1_n1_n^T+\frac{\lambda - p}{p-q} I_n  , \Ztruth\Ztrutht  - \pi\pi^T \rangle\\
& = (p-q)\langle \Ztruth\Ztrutht -I_n  , \Ztruth\Ztrutht  - \pi\pi^T \rangle\\
&\quad + (\lambda-q)\langle I_n-1_n1_n^T , \Ztruth\Ztrutht  - \pi\pi^T \rangle\\
& = (p-\lambda)\langle \Ztruth\Ztrutht -I_n  , \Ztruth\Ztrutht  - \pi\pi^T \rangle\\
&\quad + (\lambda-q)\langle  \Ztruth\Ztrutht -1_n1_n^T , \Ztruth\Ztrutht  - \pi\pi^T \rangle\\
&=2(p-q)\alpha(\pi;\Ztruth) + 2(\lambda -q)\gamma(\pi;\Ztruth).
\end{align*}
Consequently, we obtain the desired bound.
\end{proof}

\begin{proposition}\label{prop:alpha_gamma}
If $\Ztruth\in\Pi_0^{(\rho,\rho')}$, $\pi\in\Pi_1$, we have
\begin{align*}
\alpha(\pi;\Ztruth)+\gamma(\pi;\Ztruth)\geq \frac{\rho n}{16 k } \ell(\pi,\Ztruth).
\end{align*}
\end{proposition}
\begin{proof}
We use $\alpha,\gamma$ instead of $\alpha(\pi;\Ztruth),\gamma(\pi;\Ztruth)$ for simplicity. Without loss of generality we assume $\norm{\pi-\Ztruth}_1=\ell(\pi,\Ztruth)$. Define $\mathcal{C}_u=\{i:\Ztruth_{i,u}=1\}$ and $L_{u,v}=\sum_{i\in\mathcal{C}_u}\pi_{i,v}$.  We have the equality $\sum_{v}L_{u,v}=|\mathcal{C}_u|$ and also
\begin{align*}
&\alpha=\frac{1}{2}\sum_{u}\left[|\mathcal{C}_u|^2-\sum_{i,j\in\mathcal{C}_u}\sum_w\pi_{i,w}\pi_{j,w}\right]=\frac{1}{2}\sum_{u}\left[|\mathcal{C}_u|^2-\sum_wL_{u,w}^2\right]=\frac{1}{2}\sum_u\sum_{w\neq w'}L_{u,w}L_{u,w'}\\
\text{and }&\gamma=\frac{1}{2}\sum_{u\neq v}\sum_{i\in\mathcal{C}_u,j\in\mathcal{C}_v}\sum_w\pi_{i,w}\pi_{j,w}=\frac{1}{2}\sum_{u\neq v}\sum_wL_{u,w}L_{v,w}.
\end{align*}
We define $[k]$ into two disjoint subsets $S_1$ and $S_2$ where
\begin{align*}
&S_1=\Big\{u\in[k]:\forall v\neq u, L_{u,v}\leq \frac{3}{4}|\mathcal{C}_u|\Big\},\\
\text{ and }&S_2=\Big\{i\in[k]:\exists v\neq u, L_{u,v}> \frac{3}{4}|\mathcal{C}_u|\Big\}.
\end{align*}
Define $L_u=\sum_{v\neq u}L_{u,v}$. For any $u\in S_1$, if $L_{u,u}\geq |\mathcal{C}_u|/4$, we have $|\mathcal{C}_u|^2-\sum_wL_{u,w}^2\geq L_{u,u}L_u\geq |\mathcal{C}_u|L_u/4$. If $L_{u,u}< \frac{1}{4}|\mathcal{C}_u|$ we have $|\mathcal{C}_u|^2-\sum_wL_{u,w}^2\geq \frac{3}{8}|\mathcal{C}_u|^2\geq |\mathcal{C}_u|L_u/4$ as well. This leads to
\begin{align*}
\alpha\geq \frac{1}{2}\sum_{u\in S_1}\left[|\mathcal{C}_u|^2-\sum_wL_{u,w}^2\right]\geq \frac{1}{8}\sum_{u\in S_1}|\mathcal{C}_u|L_u.
\end{align*}
For any $u\in S_2$ there exists a $v\neq u$ such that $L_{u,v}> \frac{3}{4}|\mathcal{C}_u|$. We must have $L_{u,u}+L_{v,v}\geq L_{u,v}+L_{v,u}$ otherwise $\norm{\pi-\Ztruth}_1=\ell(\pi,\Ztruth)$ does not hold since we can switch the $u$-th and $v$-th columns of $\pi$ to make $\norm{\pi-\Ztruth}_1$ smaller. Consequently, we have $L_{v,v}\geq L_u/2$. So we have $\sum_{u'\neq u}\sum_wL_{u,w}L_{u',w}\geq L_{u,v}L_{v,v}\geq 3|\mathcal{C}_u|L_u/8$. Then we have
\begin{align*}
\gamma \geq \frac{1}{2}\sum_{u\in S_2}\sum_{u'\neq u}\sum_wL_{u,w}L_{u',w}\geq \frac{3}{8}\sum_{u\in S_2}|\mathcal{C}_u|L_u.
\end{align*}
Thus,
\begin{align*}
\alpha +\gamma\geq \frac{1}{16}\sum_{u}|\mathcal{C}_u|L_u \geq \frac{\rho n}{16 k}\sum_u L_u\geq \frac{\rho n}{16 k }\norm{\pi-\Ztruth}_1=\frac{\rho n}{16 k } \ell(\pi,\Ztruth).
\end{align*}
\end{proof}

\subsection{Statements and Proofs of Lemmas and Propositions for Theorem \ref{thm:gibbs}}

\begin{lemma}\label{lem:concentration_beta}
Let $X\sim \tbeta(\alpha,\beta)$ where $\alpha = n^2 p$ and $\beta =n^2(1-p)$ with $p=o(1)$. Let $\eta=o(1)$. Then we have
\begin{align*}
\p(|X-p|\geq  \eta p)\leq \exp(-\eta^2 n^2p/2).
\end{align*}
\end{lemma}
\begin{proof}
Note $X$ has the same distribution as $Y/(Y+Z)$ where $Y$ and $Z$ are independent $\chi^2$ random variables with $Y\sim\chi^2(2\alpha)$ and $Z\sim \chi^2(2\beta)$. Then by using tail bound of $\chi^2$ distribution (i.e., Proposition \ref{prop:concentration_chi}) 
\begin{align*}
\p(|X-p|\geq  \eta p)&\leq \p(|Y - 2n^2p|\geq  2\eta n^2p) + \p(|Y+Z - 2n^2| \geq \eta n^2)\\
&\leq 2\exp(-\eta^2 n^2p/4) + 2\exp(-\eta^2n^2/16)\\
&\leq \exp(-\eta^2 n^2p/2).
\end{align*}
\end{proof}

\begin{proposition}\label{prop:concentration_chi}
Let $X\sim \chi^2(k)$ we have
\begin{align*}
\p \Big(|X-k|\geq kt\Big)\leq  2\exp(-kt^2/8),\forall t\in(0,1).
\end{align*}
\end{proposition}
\begin{proof}
See Lemma 1 of \cite{laurent2000adaptive}.
\end{proof}

\section{General Derivations of CAVI for Variational Inference}\label{sec:appendix_general}
In this section, we provide the derivation from Equation (\ref{eqn:CAVI_update_general}) to Equation (\ref{eqn:CAVI_update_general_explicit}). First we have
\begin{align}\label{eqn:MF_simplification_proof}
\text{KL}(\mathbf{q}(x)\|\mathbf{p}(x|y))&=\E_{\mathbf{q}(x)}\Big[\log \frac{\mathbf{q}(x)}{\mathbf{p}(x|y)}\Big]\\
&=\E_{\mathbf{q}(x)}[\log \mathbf{q}(x)]-\E_{\mathbf{q}(x)}[\log \mathbf{p}(x|y)]\nonumber\\
&=\E_{\mathbf{q}(x)}[\log \mathbf{q}(x)]-\E_{\mathbf{q}(x)}[\log \mathbf{p}(x, y)]+\log \mathbf{p}(y)\nonumber\\
&=-(\E_{\mathbf{q}(x)}[\log \mathbf{p}(x, y)]-\E_{\mathbf{q}(x)}[\log \mathbf{q}(x)])+\log \mathbf{p}(y)\nonumber\\
&=-\left[\E_{\mathbf{q}(x)}[\log \mathbf{p}(y|x)]-\kl(\mathbf{q}(x)\|\mathbf{p}(x))\right] +\log \mathbf{p}(y).\nonumber
\end{align}
Thus, to minimize $\text{KL}(\mathbf{q}(x)\|\mathbf{p}(x|y))$ w.r.t. $\mathbf{q}(x)$ is equivalent to maximize $ \E_{\mathbf{q}(x)}[\log \mathbf{p}(y|x)]-\kl(\mathbf{q}(x)\|\mathbf{p}(x))$. 

Recall we have independence under both $\mathbf{p}$ and $\mathbf{q}$ for $\{x_i\}_{i=1}^n$. For simplicity, denote $x_{-i}$ to be $\{x_j\}_{j\neq i}$ and $\mathbf{q}_{-i}$ to be $\prod_{j\neq i}\mathbf{q}_j$. We have the decomposition
\begin{align*}
b_i(\mathbf{q}_i) &\triangleq\E_{\mathbf{q}(x)}[\log \mathbf{p}(x, y)]-\E_{\mathbf{q}(x)}[\log \mathbf{q}(x)]\\
&=\E_{\mathbf{q}_i}\left[\E_{\mathbf{q}_{-i}}\left[\log \mathbf{p}(x_i,x_{-i}, y)\right]\right]-\E_{\mathbf{q}_i}\left[\E_{\mathbf{q}_{-i}}[\log \mathbf{q}(x_i,x_{-i})]\right]\\
&=\E_{\mathbf{q}_i}\left[\E_{\mathbf{q}_{-i}}\left[\log \mathbf{p}(x_i|x_{-i}, y)\right]\right]-\E_{\mathbf{q}_i}\left[\log \mathbf{q}_i(x_i)\right] + \text{const}\\
& = - \E_{\mathbf{q}_i} \log \frac{\log \mathbf{q}_i(x_i)}{c^{-1}\exp\left[\E_{\mathbf{q}_{-i}}\left[\log \mathbf{p}(x_i|x_{-i}, y)\right]\right]} + \text{const},
\end{align*}
where the constant includes all terms not depending on $x_i$ and $c=\sum_{x_i}\exp\left[\E_{\mathbf{q}_{-i}}\left[\log \mathbf{p}(x_i|x_{-i}, y)\right]\right]$ which is also independent of $x_i$. It is obvious that to solve Equation (\ref{eqn:CAVI_update_general}) is equivalent to
\begin{align*}
\mathbf{\hat q}_i&=\argmax_{\mathbf{q}_i} b_i(\mathbf{q}_i)\\
&=\argmin_{\mathbf{q}_i} \kl\left[\mathbf{q}_i\|c^{-1}\exp\left[\E_{\mathbf{q}_{-i}}\left[\log \mathbf{p}(x_i|x_{-i}, y)\right]\right]\right].
\end{align*}
Immediately we have $\mathbf{\hat q}_i(x_i) = c^{-1}\exp\left[\E_{\mathbf{q}_{-i}}\left[\log \mathbf{p}(x_i|x_{-i}, y)\right]\right]$.
Or we may write it as
\begin{align*}
\mathbf{\hat q}_i(x_i)\propto\exp\left[\E_{\mathbf{q}_{-i}}\left[\log \mathbf{p}(x_i|x_{-i}, y)\right]\right].
\end{align*}

\end{document}